\documentclass[10pt]{article}
\usepackage{amsmath} %
\usepackage{amssymb}

\usepackage[T1]{fontenc}
\usepackage[dvips]{color}
\usepackage{subfigure}
\usepackage{multicol}
\usepackage{graphicx}
\definecolor{dgreen}{rgb}{0,.8,.3}
\definecolor{lblue}{rgb}{.2,.3,.7}

\newtheorem{assumption}{Assumption}
\newtheorem{theorem}{Theorem}

\newtheorem{lemma}{Lemma}
\newtheorem{proposition}{Proposition}

\newtheorem{remark}{Remark}

\numberwithin{equation}{section}
\numberwithin{lemma}{section}
\numberwithin{theorem}{section}

\newcommand{\beq}{\begin{equation}}
\newcommand{\eeq}{\end{equation}}
\newcommand{\beqa}{\begin{eqnarray}}
\newcommand{\eeqa}{\end{eqnarray}}
\newcommand{\beqas}{\begin{eqnarray*}}
\newcommand{\eeqas}{\end{eqnarray*}}
\newcommand{\ba}{\begin{array}}
\newcommand{\ea}{\end{array}}
\newcommand{\bi}{\begin{itemize}}
\newcommand{\ei}{\end{itemize}}
\newcommand{\gap}{\hspace*{2em}}
\newcommand{\sgap}{\hspace*{1em}}
\newcommand{\nn}{\nonumber}

\def\arg{{\rm arg}}
\def\Arg{{\rm Arg}}
\def\bA{{\bar A}}
\def\bb{{\bar b}}
\def\bF{{\bar F}}

\def\c{{\rm c}}
\def\cC{{\cal C}}
\def\cg{{\rm cg}}

\def\cI{{\cal I}}

\def\cH{{\cal H}}
\def\cK{{\cal K}}

\def\cP{{\cal P}}
\def\cS{{\cal S}}

\def\dist{{\rm dist}}
\def\eps{{\epsilon}}
\def\bcH{{\bar \cH}}

\def\hA{{\hat A}}
\def\hb{{\hat b}}
\def\hell{{\hat \ell}}
\def\heps{{\hat \epsilon}}
\def\heta{{\hat \eta}}
\def\hf{{\hat f}}
\def\hF{{\hat F}}
\def\hp{{\hat p}}
\def\hr{{\hat r}}

\def\hcH{{\hat \cH}}

\def\hx{{\hat x}}

\def\low{{\rm low}}
\def\range{{\rm Range}}
\def\rank{{\rm rank}}

\def\sgn{{\rm sgn}}
\def\Span{{\rm span}}
\def\tc{{\rm tc}}

\def\tF{{\tilde F}}

\def\tk{{k}}
\def\tx{{\tilde x}}
\def\CG{{\rm GCG}}
\def\TPCG{{\rm TPCG}}
\def\v{{\rm v}}
\def\vp{{v^{\rm p}}}
\def\bcH{{\bar \cH}}
\def\L{{\cal L}}
\def\cK{{\cal K}}

\title{Generalized Conjugate Gradient Methods for $\ell_1$ Regularized
Convex Quadratic Programming with Finite Convergence}
\author{
\ Zhaosong Lu\thanks{Department of Mathematics, Simon Fraser University, Canada. Email: {\tt zhaosong@sfu.ca}. This author's work was supported in part by NSERC.} \ and Xiaojun Chen\thanks{Department of Applied Mathematics, The Hong Kong Polytechnic University, Hung Hom, Kowloon, Hong Kong. Email: {\tt maxjchen@polyu.edu.hk}. The author's work was supported in part by Hong Kong Research Council Grant PolyU153001/14p.}}

\date{November 24, 2015 (Revised: February 12, 2016)}

\begin{document}

\maketitle

\begin{abstract}
The conjugate gradient (CG) method is an efficient iterative method for
solving large-scale strongly convex quadratic programming (QP). In this paper we  
 propose some generalized CG (GCG) methods for solving the $\ell_1$-regularized 
(possibly not strongly) convex QP that terminate at an optimal solution in a finite 
number of iterations. At each iteration, our methods first identify a face of an orthant and then either perform an exact line search along the direction of the negative projected minimum-norm subgradient of the objective function or execute a CG subroutine that conducts a sequence of CG iterations until a CG iterate crosses the boundary of this face or an approximate minimizer of over this face or a subface is found.  We determine which type of step should be taken by comparing the magnitude of some components of the minimum-norm subgradient of  the objective function to that of its rest components.  Our analysis on finite convergence of these methods makes use of an error bound result and some key properties of the aforementioned exact line search and the CG subroutine. We also show that the proposed methods are capable of finding an approximate solution of the problem by allowing some inexactness on the execution of the CG subroutine. The overall arithmetic operation cost of our GCG methods for finding an $\eps$-optimal solution depends on $\eps$ in $O(\log(1/\eps))$, which is superior to the accelerated proximal gradient method \cite{Beck,Nest13} that depends on $\eps$ in $O(1/\sqrt{\eps})$. In addition, our GCG methods can be extended straightforwardly to solve box-constrained convex QP with finite convergence. Numerical results demonstrate that our methods are very favorable for solving ill-conditioned problems. 
%Numerical results illustrating the efficiency of the GCG methods are presented.  

\end{abstract}

\vskip14pt
\noindent {\bf Keywords}: conjugate gradient method, convex quadratic programming, $\ell_1$-regularization, sparse optimization, finite convergence

\vskip14pt
\noindent {\bf AMS subject classifications}: 65C60, 65K05, 65Y20, 90C06, 90C20, 90C25

%%%%%%%%%%%%%%%%%%%%%%%%%%%%%%%%
\section{Introduction}
\label{intro}

The conjugate gradient (CG) method is an efficient numerical method for solving {\it strongly} convex quadratic programming (QP) in the form of 
\beq \label{CG-qp}
\min\limits_{x\in\Re^n} \frac12 x^T B x - c^T x,
\eeq
or equivalently, the linear system $B x=c$,  where $B \in \Re^{n\times n}$ is a symmetric {\it positive definite} matrix and $c\in \Re^n$. It terminates at  the unique      
 optimal solution of \eqref{CG-qp} in a finite number of iterations. Moreover, it is suitable for solving large-scale problems since it only requires 
matrix-vector multiplications per iteration (e.g., see \cite{NoWr06} for details). The CG method has also been generalized to minimize a convex quadratic function over a box or a ball (e.g., see \cite{Dostal1,Dostal2,Oleary,Steihaug,Toint}).

In this paper we are interested in generalizing the CG method to solve the $\ell_1$ regularized convex QP:
\beq \label{l1-qp}
F^* = \min\limits_{x\in\Re^n} F(x) := \frac12 x^T A x - b^T x + \tau \|x\|_1,
\eeq
where $A\in\Re^{n\times n}$ is a symmetric {\it positive semidefinite} matrix, $b\in \Re^n$ and $\tau \ge 0$ is a regularized parameter. Throughout this paper we make the following assumption for problem \eqref{l1-qp}.
\begin{assumption} \label{assump}
The set of optimal solutions of problem \eqref{l1-qp}, denoted by $\cS^*$, is nonempty. \footnote{Since the objective function of
\eqref{l1-qp} is a convex piecewise quadratic function, problem \eqref{l1-qp} has at least an
optimal solution if and only if its objective function is bounded below.}
\end{assumption}

Over the last decade, a great deal of attention has been focused on problem \eqref{l1-qp} due to numerous applications in image sciences, machine learning, signal processing and statistics (e.g., see  \cite{ChDoSa98,HaTiFr01,BDE,Gold,WrNoFi09,SrNoWr11} and the references therein). Considerable effort has been devoted to developing efficient algorithms for solving \eqref{l1-qp} (e.g., see \cite{Beck,Nest13,WrNoFi09,Toh,HaYiZh07,JYang,XiZh13,Ulbrich}). These methods are iterative methods and capable of producing an approximate solution to \eqref{l1-qp}. Nevertheless, they generally cannot terminate at an 
optimal solution of \eqref{l1-qp}.  Recently, Byrd et al.\ \cite{Byrd} proposed a method called iiCG to solve \eqref{l1-qp} that combines the iterative soft-thresholding algorithm  (ISTA) \cite{Beck,Dono95,WrNoFi09} with the CG method. Under the assumption that $A$ is symmetric positive definite, it was shown in \cite{Byrd} that the sequence generated by iiCG converges to the unique optimal solution of \eqref{l1-qp}, and if additionally this solution satisfies strict complementarity, iiCG terminates in a finite number of iterations. Its convergence is, however, unknown when $A$ is positive 
semidefinite (but not definite), which is typical for many instances of \eqref{l1-qp} arising in applications. 
%At each iterate $x^k$, iiCG computes a  proximal gradient $g^k$ of $F$ at $x^k$ and partition it into two subvectors $g^{k,1}$    

In this paper we propose some generalized CG ($\CG$) methods for solving \eqref{l1-qp}  that terminate at an optimal solution of \eqref{l1-qp} in a finite number of iterations with no additional assumption. At each iteration, 
our methods first identify a certain face of some orthant and then either perform an exact line search along the direction of the negative projected minimum-norm subgradient of $F$ or execute a CG  subroutine that conducts a sequence of CG iterations until a CG iteration crosses the boundary of this face or an approximate minimizer of $F$ over this face or a subface is found. The purpose of the exact line search step is to release some zero components of the current iterate so that the value of $F$ is sufficiently reduced. The aim of executing a CG routine is to update the nonzero components of the current iterate, which also results in a reduction on $F$.  We determine which type of step should be taken by comparing the magnitude of some components of the minimum-norm subgradient of $F$ to that of its rest components.  Our methods are 
substantially different from the iiCG method \cite{Byrd}. In fact, at each iteration,  iiCG either performs a {\it proximal gradient} step or executes a {\it single} CG iteration. It determines  which type of step should be conducted by comparing the magnitude of some components of a {\it proximal gradient} of $F$ to that of its rest components. 

In order to analyze the convergence of our GCG methods, we establish some error bound results for problem \eqref{l1-qp}. We also conduct some exclusive 
analysis on the aforementioned exact line search and the CG subroutine. Using these results, we show that the $\CG$ methods terminate at an optimal solution of \eqref{l1-qp} in a finite number of iterations. To the best of our knowledge, the GCG methods are the first methods for solving (\ref{l1-qp}) with finite convergence.  We also show that our methods are capable of finding an approximate solution of \eqref{l1-qp}  by allowing some inexactness on the execution of the CG subroutine. 
%To the best of our knowledge, the $\CG$ methods are the first methods for solving \eqref{l1-qp} with finite convergence. 
The overall arithmetic operation cost of our GCG methods for finding an $\eps$-optimal solution depends on $\eps$ in $O(\log(1/\eps))$, which is superior to the accelerated proximal gradient method \cite{Beck,Nest13} that depends on $\eps$ in $O(1/\sqrt{\eps})$. In addition, it shall be mentioned that these methods can be extended %straightforwardly 
to solve the following box-constrained convex QP with finite convergence:
\beq \label{bQP}
\min\limits_{l \le x \le u} \frac12 x^T A x - b^T x, 
\eeq
where $A\in\Re^{n\times n}$ is symmetric positive semidefinite, $b\in \Re^n$, $l,u\in\bar\Re^n$ with $\bar\Re=[-\infty,\infty]$. As for finite convergence, 
the existing CG type methods \cite{Dostal1,Dostal2} for \eqref{bQP}, however, require that $A$ be symmetric positive definite. The extension of our methods to problem \eqref{bQP} is not included in this paper 
due to the length limitation.
%Though the extension of our methods to problem \eqref{bQP} is not included due to the paper 
%length limitation, it perhaps can be viewed as a part of the contribution of this paper. 

The rest of the paper is organized as follows. In Section \ref{err-bdds}, we establish some results on error bound for problem \eqref{l1-qp}. In Section \ref{algorithm},  we propose several $\CG$ methods for solving problem \eqref{l1-qp} and establish 
their finite convergence. In Section \ref{l1-ls-prob}, we discuss the application of our  $\CG$ methods to solve the $\ell_1$ regularized least-squares problems and develop a practical termination criterion for them.  We conduct numerical experiments in Section \ref{result} to compare the performance of our $\CG$ methods with some 
state-of-the-art algorithms for solving problem \eqref{l1-qp}. 
% In addition, we present some concluding remarks in Section \ref{conclude}. 
In Section \ref{conclude} we present some concluding remarks. Finally, in the appendix we study some convergence properties of the standard CG method for solving (possibly not strongly) convex QP. 
 
\subsection{Notation and terminology} \label{notation}

For a nonzero symmetric positive semidefinite matrix $A$, we define a generalized condition number 
of $A$ as
\beq \label{kappa}
\kappa(A) = \|A\|\|A^+\| = \frac{\lambda_{\max}(A)}{\lambda^+_{\min}(A)},
\eeq
where $A^+$ is the Moore-Penrose pseudoinverse of $A$, $\lambda_{\max}(A)$ is the largest eigenvalue of $A$
and $\lambda^+_{\min}(A)$ is the smallest {\it positive} eigenvalue of $A$. Clearly, it reduces to 
the standard condition number when $A$ is symmetric positive definite. In addition, for any index set $J\in\{1,\ldots,n\}$, $|J|$ is the cardinality of $J$ and $A_{JJ}$ is the submatrix of $A$ formed by its rows and columns indexed by $J$. Analogously, $b_J$ is the subvector of $b\in\Re^n$ formed by its components indexed by $J$. In addition, the range space and rank of a matrix $B$ are denoted by $\range(B)$ and $\rank(B)$, 
respectively. 

Let $\sgn: \Re^n \to \{-1,0,1\}^n$ be  the standard sign operator, which is conventionally defined as follows
\[
[\sgn(x)]_i = \left\{
\ba{ll}
1 & \mbox{if} \ x_i >0; \\
0 & \mbox{if} \ x_i =0; \\
-1 & \mbox{if} \ x_i <0,
\ea
\right. \quad i=1,\dots. n.
\] 
Let $F$ be defined in \eqref{l1-qp} and 
\beq \label{f}
f(x) = \frac12 x^T A x - b^T x.
\eeq 
Let $v(x)$ be the {\it minimum-norm subgradient} of $F$ at $x$, which is 
the projection of the zero vector onto the subdifferential of $F$ at $x$. It follows that
\beq \label{v}
v_i(x) = \left\{
\ba{ll}
\nabla_i f(x)+ \tau \ \sgn(x_i) & \ \mbox{if} \ x_i \neq 0; \\ [5pt]
\min\left(\nabla_i f(x) + \tau, \max(0, \nabla_i f(x)-\tau)\right) &  \ \mbox{if} \ x_i=0,
\ea
\right.  \quad i=1,\ldots, n,
\eeq
where $\nabla_i f(x)$ denotes the $i$th partial derivative of $f$ at $x$. 
It is known that $x$ is an optimal solution of problem \eqref{l1-qp} if and only if $0\in \partial F(x)$, where $\partial F$ denotes the subdifferential of $F$.  Since $0\in \partial F(x)$
is equivalent to $v(x)=0$, $x$ is an optimal solution of \eqref{l1-qp} if and only if $v(x)=0$. 

For any $x\in\Re^n$, we define
\beq \label{index-I}
\ba{l}
I_-(x) = \{i: x_i<0\},  \quad  I_+(x) = \{i: x_i>0\}, \\ [4pt]
 I_0(x) = \{i: x_i=0\}, \quad  I^\c_0(x) =  \{i: x_i \neq 0\}, 
\ea
\eeq
and also define
\beq \label{Fxs}
\cH(x) = \{y\in\Re^n:  y_i=0, \ i\in I_0(x)\}, \quad\quad F^*_x = \min\{F(y): y \in \cH(x)\}.
\eeq
 In addition, given any closed set $S\subseteq \Re^n$, ${\rm dist}(x,S)$ denotes the distance from $x$ to $S$, and $\cP_S(x)$ denotes the projection of $x$ to $S$. Finally, we define 
\beq
\cI^* = \{J\subseteq I_0(x^*): x^*\in \cS^*\},  \quad \quad \L(n) = \max\left\{\ell: 
\cC_i \notin \cI^*, \ i=1,\ldots, \ell,  \ \mbox{are distinct subsets in} \  \{1,\ldots,n\}  
\right\}+1.\label{L}
\eeq
% \beqa
%\cI^* &=& \{J\subseteq I_{x^*}: x^*\in \cS^*\}, \label{cIs} \\
%\L(n) &=& \max\left\{\ell: 
%\cC_i \notin \cI^*, \ i=1,\ldots, \ell,  \ \mbox{are distinct subsets in} \  \{1,\ldots,n\}  
%\right\}.
%\eeqa
%\beq\label{L}
%\L(n) = \max\left\{\ell: \ba{l} 
%\cC_i, \ i=1,\ldots, \ell,  \ \mbox{are subsets in} \  \{1,\ldots,n\} \\ \mbox{and}  \
%\cC_i \not\subseteq\ \cC_j  \ \mbox{for all} \ i, j =1, \ldots, \ell, \  j >i.    
%\ea\right\}.
%\eeq
%\beq\label{L}
%\L(n) = \max\left\{\ell: \ba{l} 
%\cC_i, \ i=1,\ldots, \ell,  \ \mbox{are distinct subsets in} \  \{1,\ldots,n\} \ \mbox{and}  \\
%\cC_i \not\subseteq\ \cC_j  \ \mbox{for all} \ i, j =1, \ldots, \ell \ \mbox{and} \ j >i \ \mbox{if} \ \cC_i \neq \emptyset.    
%\ea\right\}.
%\eeq

%%%%%%%%%%%%%%%%%%%%%%%%%%%%%%%%%%%%%%%

%\section{Technical preliminaries}
%\label{tech}
%
%\subsection{Error bounds} 
\section{Error bound results}
\label{err-bdds}

In this section we develop some error bound results for problem \eqref{l1-qp}.  
To proceed, let $\cS(\delta) := \{x: F(x)-F^* \le \delta\}$ for any $\delta \ge 0$, where  $F$ and $F^*$ are defined in \eqref{l1-qp}. We first bound the gap between $F(x)$ and $F^*$ by $\|v(x)\|$ for all $x\in \cS(\delta)$.

\begin{theorem}  \label{errbdd-F}
Let $F$, $F^*$ and $v$ be defined in \eqref{l1-qp} and \eqref{v}, respectively.  Then for any $\delta \ge 0$, there exists some 
$\eta>0$ (depending on $\delta$) such that 
\[
F(x) - F^* \le \eta \|v(x)\|^2, \quad\quad \forall x  \in \cS(\delta).
\] 
\end{theorem}

\begin{proof}
Let $X^*$ denote the set of optimal solutions of \eqref{l1-qp}. Notice that $F$ is a convex 
piecewise quadratic function. By \cite[Theorem 2.7]{Li95}, there exists some $\eta>0$ such 
that 
\beq \label{err-bdd-F}
\dist(x,X^*) \le \ \sqrt{\eta} \sqrt{F(x)-F^*}, \quad\quad \forall x  \in \cS(\delta).
\eeq
%Let $x^*\in X^*$ be such that $\|x-x^*\|=\dist(x,X^*)$. It then follows that 
%\beq \label{err-bdd-F}
%\|x-x^*\|  \le \sqrt{\eta}\sqrt{F(x)-F^*}, \quad\quad \forall x  \in \cS(\delta).
%\eeq 
Let $x^*\in X^*$ be such that $\|x-x^*\|=\dist(x,X^*)$. By $v(x) \in \partial F(x)$ and the convexity of $F$, one has 
\[
 F(x) - F^* = F(x)-F(x^*) \le  \langle v(x), x -x^* \rangle \le \|v(x)\| \|x-x^*\| = \|v(x)\| \dist(x,X^*),
\] 
which together with \eqref{err-bdd-F} implies that the conclusion holds.
%\[
% F(x) - F^* \le \sqrt{\eta} \|v(x)\|\sqrt{F(x)-F^*}, \quad\quad \forall x  \in \cS(\delta),
%\]
%which immediately leads to the conclusion.
\end{proof}

\gap

We next bound the gap between $F(x)$ and $F^*_x$ by the magnitude of some components of $v(x)$ for all $x\in \cS(\delta)$.

\begin{theorem}\label{errbdd-Fx}
Let $F$ and $F^*_x$ be defined in \eqref{l1-qp} and \eqref{Fxs}, respectively. Then for any $\delta \ge 0$, 
there exists some $\hat \eta>0$ (depending on $\delta$) such that 
\[
F(x) - F^*_x \le \hat \eta \|[v(x)]_J\|^2, \quad\quad \forall x \in \cS(\delta),
\]
where $J=I^\c_0(x)$.
\end{theorem}

\begin{proof}
Let $x\in\cS(\delta)$ be arbitrarily chosen, $I=I_0(x)$ and $J=I^\c_0(x)$. If $J=\emptyset$, it is clear that $x=0$ and hence $F^*_x=F(x)$. Also, by convention $\|[v(x)]_J\|=0$. These imply the conclusion holds. We now assume $J\neq\emptyset$. Consider the problem 
\beq \label{hFJs}
\hF^*_J = \min\limits_{z\in \Re^{|J|}} \hF_J(z) := \frac12  z^T A_{JJ} z - b^T_J z
+ \tau \|z\|_1.
\eeq 
  In view of the definitions of $F^*_x$, $\hF_J$, $\hF^*_J$, $F$, $F^*$ and $J$, one can observe that 
\[
\hF_J(x_J)=F(x), \quad\quad \hF^*_J=F^*_x \ge F^*.
\]
This together with $x\in\cS(\delta)$ implies that $\hF_J(x_J)-\hF^*_J  \le  F(x)-F^* \le \delta$.  By \eqref{v}, \eqref{hFJs} and the definition of $J$, we also observe that $[v(x)]_J$ is the minimum-norm subgradient of $\hF_J$ at $x_J$. In addition, notice that problem \eqref{hFJs} is in the same form as \eqref{l1-qp}. By these facts and applying Theorem \ref{errbdd-F} to  problem \eqref{hFJs}, there exists some $\eta_J>0$ (depending on $\delta$ and $J$) such that 
\beq \label{F-bdd-J}
F(x)-F^*_x =  \hF_J(x_J) - \hF^*_J \le \eta_J \|[v(x)]_J\|.
\eeq 
Let $\hat\eta = \max\{\eta_J: J=I^\c_0(x), \ x\in\cS(\delta)\}$, which is finite due to the fact that all possible choices of $J$ are finite. The conclusion immediately follows from this and \eqref{F-bdd-J}.
\end{proof}

\gap

%Before ending this section we establish some relationship between objective value gap and the magnitude of
%a gradient for problem \eqref{qp}.

The error bound presented in Theorem \ref{errbdd-Fx} is a local error bound as it depends on $\delta$. In addition, Theorem  \ref{errbdd-Fx} only ensures the existence of some parameter $\eta$ for the error bound, but its actual value is generally unknown. We next derive a global error bound with a known $\eta$ for problem \eqref{l1-qp} when $A$ is symmetric positive definite. To proceed, we first establish a lemma as follows.

\begin{lemma} \label{gap-relation}
Suppose $A \neq 0$ and $b\in\range(A)$. Let $f(x)$ be defined in \eqref{f} and 
$f^*=\min_{x\in\Re^n} f(x)$. Then there holds:
\[
\frac{1}{2\|A\|} \|\nabla f(x)\|^2  \le f(x) - f^* \le \frac{\|A^+\|}{2} \|\nabla f(x)\|^2, \quad\quad \forall x\in\Re^n.
\]
\end{lemma}

\begin{proof}
Let $\lambda_1 \ge \lambda_2 \ge \cdots \ge \lambda_n \ge 0$ be all eigenvalues of $A$ and $\{u_i\}^n_{i=1}$ the corresponding orthonormal eigenvectors. In addition, let $x^*$  be
an optimal solution of the problem $\min_{x\in\Re^n} f(x)$. Clearly,  $Ax^*=b$. Moreover, for any $x\in\Re^n$,  we have $ x-x^* = \sum^n_{i=1} \alpha_i u_i$ for some $\{\alpha_i\}^n_{i=1}$. These imply
\beq \label{expand-g}
\nabla f(x) = Ax-b = A(x-x^*) = \sum^n_{i=1} \lambda_i \alpha_i u_i.
\eeq
 Let $\ell=\rank(A)$. It follows that $\lambda_i=0$ for all $i>\ell$. In view of this and \eqref{expand-g}, we have
\[
\|\nabla f(x)\|^2 = \sum^n_{i=1} \lambda_i^2 \alpha^2_i =  \sum^\ell_{i=1} \lambda_i^2 \alpha^2_i.
\]
This together with the fact $\lambda_1 \ge \cdots \ge \lambda_\ell >0$ yields
\[
\frac{1}{\lambda_1} \|\nabla f(x)\|^2 = \frac{1}{\lambda_1} \sum^\ell_{i=1} \lambda^2_i \alpha^2_i \ \le \   \sum^\ell_{i=1} \lambda_i \alpha^2_i
\ \le \ \frac{1}{\lambda_\ell} \sum^\ell_{i=1} \lambda^2_i \alpha^2_i = \frac{1}{\lambda_\ell} \|\nabla f(x)\|^2.
\]
Using the definitions of $f$ and $x^*$, \eqref{expand-g}, $x-x^* = \sum^n_{i=1} \alpha_i u_i$ and $\lambda_i=0$ for all $i>\ell$, one can observe that
\[
f(x) - f^* = \frac12 (x-x^*)^T A(x-x^*) = \frac12 \sum^n_{i=1} \lambda_i \alpha^2_i =  \frac12 \sum^\ell_{i=1} \lambda_i \alpha^2_i.
\]
The conclusion then immediately follows from the last two relations and 
the fact that $\lambda_1=\|A\|$ and $\lambda_\ell=1/\|A^+\|$. 
\end{proof}

\begin{theorem}  \label{errbdd-F}
Let $F$ and $F^*_x$ be defined in \eqref{l1-qp} and \eqref{Fxs}, respectively. 
Suppose that $A$ is symmetric positive definite. Then there holds:
\[
F(x) - F^*_x \le \frac{\|A^{-1}\|}{2} \|[v(x)]_J\|^2,  \quad\quad \forall x\in\Re^n,
\]
where $J=I^\c_0(x)$.
\end{theorem}

\begin{proof}
Let $x\in\Re^n$ be arbitrarily chosen and let $J=I^\c_0(x)$. If $J=\emptyset$, it is clear that $x=0$ and hence $F^*_x=F(x)$. Also, by convention $\|[v(x)]_J\|=0$. These imply the conclusion holds. We now assume $J\neq\emptyset$. Consider the problem 
\[
\tF^*_J = \min\limits_{z\in \Re^{|J|}} \tF_J(z) := \frac12  z^T A_{JJ} z + (-b_J+\sgn(x_J))^T z.
\]
Since $A$ is positive definite, so is $A_{JJ}$. It then follows that 
$\sgn(x_J)-b_J\in\range(A_{JJ})$. By applying Lemma \ref{gap-relation} to this problem,  we obtain that 
\beq \label{gap-rel}
\tF_J(x_J) - \tF^*_J \le \frac{\|(A_{JJ})^{-1}\|}{2} \|\nabla \tF_J(x_J)\|^2.
\eeq 
In addition, by the definitions of $F$, $\tF_J$ and $J$, one can observe that $\tF_J(y_J) \le F(y)$ for all $y\in \cH(x)$, 
where $\cH(x)$ is defined in \eqref{Fxs}. This together with the definitions of $\tF^*_J$ and $F^*_x$ implies $\tF^*_J \le F^*_x$. Also, we observe that $\tF_J(x_J)=F(x)$ and $[v(x)]_J = \nabla \tF_J(x_J)$. 
Using these relations and \eqref{gap-rel}, we have
\[
F(x)-F^*_x \ \le\ \tF_J(x_J) - \tF^*_J  \ \le \  \frac{\|(A_{JJ})^{-1}\|}{2} \|\nabla \tF_J(x_J)\|^2 \ \le \ 
\frac{\|A^{-1}\|}{2} \|[v(x)]_J\|^2,
\]
and hence the conclusion holds.
\end{proof}

\section{Generalized conjugate gradient  methods for \eqref{l1-qp}} 
\label{algorithm}

In this section we propose several $\CG$ methods for solving problem \eqref{l1-qp}, which terminate at an optimal solution in a finite number iterations.  A key ingredient of these methods is to apply a truncated projected CG (TPCG) method to a sequence of convex 
QP over certain faces of some orthants in $\Re^n$.  

\subsection{Truncated projected conjugate gradient methods} 

In this subsection we present two TPCG methods for finding an (perhaps very roughly) approximate solution to a convex QP on a face of some orthant in $\Re^n$ in the form of
\beq \label{sub-qp}
\ba{ll}
\min\limits_x & q(x) := f(x) + c^T x \\
\mbox{s.t.} & x_j = 0, \quad j \in J_0, \\
& x_j \le 0,  \quad j \in J_-, \\
& x_j \ge 0,  \quad  j \in J_+, \\
\ea
\eeq
where $f$ is defined in \eqref{f}, $c\in \Re^n$, and $J_-,J_0,J_+ \subseteq \{1,\ldots,n\}$ form a partition of
$\{1,\ldots,n\}$. For convenience of presentation, we denote by  $\Omega$ the feasible region of \eqref{sub-qp}. 

For the first TPCG method, each iterate is obtained by applying the standard projected CG (PCG) method   
\footnote{The PCG method applied to problem \eqref{sub-qp-eq} is equivalent to the CG method applied to the problem $\min q(x_{J^\c_0},0)$, where $J^\c_0$ is the complement of $J_0$ in $\{1,\ldots,n\}$.}  to the problem
\beq \label{sub-qp-eq}
\min\limits_x \{q(x): x_j = 0, \ j \in J_0\}
\eeq
until an approximate solution of \eqref{sub-qp-eq} is found or a PCG iterate crosses the boundary of $\Omega$.
In the former case, the method outputs the resulting approximate solution. But in the latter case, it outputs
the intersection point between the boundary of $\Omega$ and 
the line segment joining the last two PCG iterates.
% Let
%$\cH = \{x\in\Re^n: x_j = 0, \ j \in J_0\}$. Recall that an iteration of the PCG method for \eqref{sub-qp-eq}
%is given by
%\[
%\ba{l}
%x^{k+1} = x^k + \alpha_k d^k, \ \mbox{where} \ \alpha_k = \frac{\|p^k\|^2}{(d^k)^TAd^k}; \\ [5pt]
%r^{k+1} = r^k + \alpha_k A d^k; \\ [5pt]
%p^{k+1} = \cP_\cH(r^{k+1}); \\ [5pt]
%d^{k+1} = -p^{k+1} + \frac{\|p^{k+1}\|^2}{\|p^k\|^2} d^k,
%\ea
%\]
%where initially $r^0 = Ax^0-b+c$, $p^0=\cP_\cH(r^0)$, $d^0=-p^0$.
Let $x^0$ be an arbitrary feasible point of problem \eqref{sub-qp} and $\eps\ge 0$ be given.  We now present the first TPCG method for problem \eqref{sub-qp}.

\gap

\vspace{.1in}

\noindent {\bf Subroutine 1:} $y = \TPCG1(A,b,c,J_0,J_-,J_+,x^0,\eps)$
%\footnote{$x^0$ is a feasible point of problem \eqref{sub-qp} and  $\eps$ is a nonnegative number.}
% \footnote{The index sets $J_-$ and $J_+$ are implicitly
%used to maintain the generated sequence $\{x^k\}$ to stay in the feasible region of problem \eqref{sub-qp}.}

\vspace{.1in}

\noindent{\bf Input:} $A$, $b$, $c$, $J_0$, $J_-$, $J_+$, $x^0$, $\eps$.

\vspace{.1in}

\noindent Set $r^0 = Ax^0-b+c$, $\cH = \{x\in\Re^n: x_j = 0, \ j \in J_0\}$, $p^0=\cP_\cH(r^0)$, $d^0=-p^0$, $k=0$.

\vspace{.1in}

\noindent {\bf Repeat}
\bi
\item[1)] $\alpha_k = \min\{\alpha^\cg_k,\alpha^\tc_k\}$, where
\[
\ba{l}
\alpha^\cg_k = \frac{\|p^k\|^2}{(d^k)^TAd^k},\footnotemark \quad
\alpha^{\rm tc}_k = \max\{\alpha: x^k + \alpha d^k \in \Omega\}.
\ea
\]
 \footnotetext{If $\|p^k\| \neq 0$ and $(d^k)^TAd^k=0$, we set $\alpha^{\cg}_k=\infty$.}
\item[2)] $x^{k+1} = x^k + \alpha_k d^k$.
\item[3)] If $\alpha^\cg_k > \alpha_k$, return $y=x^{k+1}$ and terminate.
\item[4)] $r^{k+1} = r^k + \alpha_k A d^k$.
\item[5)] $p^{k+1} = \cP_\cH(r^{k+1})$. If $\|p^{k+1}\|_\infty \le \eps$, return
$y=x^{k+1}$ and terminate.
\item[6)] $d^{k+1} = -p^{k+1} + \frac{\|p^{k+1}\|^2}{\|p^k\|^2}d^k$.
\item[7)] $k \leftarrow k+1$.
\ei
{\bf Output:} $y$.

\gap

{\bf Remark 1:}  The iterations of the above TPCG method are almost 
identical to those of PCG applied to problem \eqref{sub-qp-eq} 
except that the step length $\alpha_k$ is chosen to be an intermediate 
one when an iterate of PCG crosses the boundary of $\Omega$. In addition,  if $\alpha^\cg_k > \alpha_k$ holds at some $k$, 
the output $y$ is on the boundary of $\Omega$. If $\|p^{k+1}\|_\infty \le \eps$ holds at some $k$, 
the output $y$ is an approximate optimal solution of problem \eqref{sub-qp}.

\gap

We next show that under a mild assumption the above method terminates 
in a finite number of iterations.

\begin{theorem} \label{TPCG1-converge}
Assume that problem \eqref{sub-qp} has at least an optimal solution. Suppose that $x^0$ is a feasible point of problem \eqref{sub-qp} and $\eps \ge 0$. Let 
$B=A_{J^\c_0J^\c_0}$, where $J^\c_0$ is the complement of $J_0$ in $\{1,\ldots,n\}$. The following  statements hold:
\bi
\item[(i)] If problem \eqref{sub-qp-eq} is bounded below, Subroutine 1 terminates in at most $\min(M,\rank(B))$ \footnote{By convention, we define $\log 0 = -\infty$. It follows  that $M=-\infty$ and hence $\min(M,\rank(B))=\rank(B)$ when $\eps=0$.} iterations, where   
\[
M = \left\lceil\frac{\log\eps - \log(2\sqrt{\kappa(B)}\|p^0\|)}{\log(\kappa(B)-1)-\log(\kappa(B)+1)}\right\rceil. 
\]
\item[(ii)] If problem \eqref{sub-qp-eq} is unbounded below, Subroutine 1 terminates in at most $\rank(B)+1$ iterations.
\ei
\end{theorem}

\begin{proof}
(i) Assume that problem \eqref{sub-qp-eq} is bounded. Suppose for contradiction that Subroutine 1 does not terminate in $\min(M,\rank(B))$ iterations. Then the iterates $x^k$, $k=1,\ldots,\min(M,\rank(B))$ of Subroutine 1 are identical to those generated by the PCG method 
applied to problem \eqref{sub-qp-eq}. Let $q^*$ denote the optimal value of \eqref{sub-qp-eq}. It follows from Theorem \ref{CG-convergence} (iii) that 
for $k=1,\ldots,\min(M,\rank(B))$, 
\[
q(x^k)-q^* \le 4 \left(\frac{\sqrt{\kappa(B)}-1}{\sqrt{\kappa(B)}+1}\right)^{2k}\left(q(x^0)-q^*\right).
\]
By the definition of $p^k$ and Lemma \ref{gap-relation}, 
we have 
\[
\|p^{k}\|^2  \le 2 \|B\| (q(x^k) - q^*), \quad\quad  q(x^0)-q^* \le \|B^+\|\|p^0\|^2/2.
\]
Using these relations, we obtain that 
\[
 \|p^{k}\|^2  \le 4\kappa(B) \left(\frac{\sqrt{\kappa(B)}-1}{\sqrt{\kappa(B)}+1}\right)^{2k}\|p^0\|^2.
\]
In view of this and Theorem \ref{cg-unbdd} (i), one can easily conclude that 
the PCG method must terminate at $x^k$ satisfying $\|p^k\|_\infty \le \eps$ for some $0 \le k \le \min(M,\rank(B))$. This contradict the above supposition.

(ii) Assume that problem \eqref{sub-qp-eq} is unbounded. Suppose for contradiction that Subroutine 1 does not terminate in $\rank(B)+1$ iterations. Then the iterates $x^k$, $k=1,\ldots,\rank(B)+1$, of Subroutine 1 are identical to those generated by the PCG method applied to problem \eqref{sub-qp-eq}. By Theorem \ref{cg-unbdd} (ii), there must exist some $0\le i \le \rank(B)+1$ such that $q(x^i+\alpha d^i) \to -\infty$ as $\alpha \to 
\infty$. Recall that $x^l$ is in $\Omega$ and problem \eqref{sub-qp} has at 
least an optimal solution. Thus there exists a least $\alpha \ge 0$ such that $x^{i+1} =x^i+\alpha d^i$ lies on the boundary of $\Omega$ and Subroutine 1 thus terminates at iteration $i$, which is a contradiction to the above supposition. 
\end{proof}

\vspace{.1in}

{\bf Remark 2:} It follows from Theorem \ref{TPCG1-converge} that when $\eps=0$, $\TPCG1$ executes at most (but possibly much less than) $n+1$ PCG iterations. On the other hand, when $\eps>0$, the number 
of PCG iterations executed in $\TPCG1$ depends on $\eps$ 
in $O(\log(1/\eps))$.

\vspace{.03in}

As seen from step 3) of Subroutine 1, it immediately terminates once an iterate crosses the boundary of $\Omega$. In this case, the output $y$ may be a rather poor approximate solution to problem \eqref{sub-qp}. In order to improve the quality of $y$, we resort an active set approach by iteratively applying Subroutine 1 to minimize $q$ over a decremental subset of $\Omega$, which is formed by incorporating the active constraints of the iterate obtained from the immediately preceding execution of Subroutine 1. 
%Before proceeding, we introduce some notations
%that will be used shortly. Given $x\in\Re^n$, we define
%\beq \label{index-I}
%I_-(x) = \{i: x_i<0\}, \quad I_0(x) = \{i: x_i=0\}, \quad  I_+(x) = \{i: x_i>0\},  \\ [5pt]
%\eeq
Let $x^0$ be an arbitrary feasible point of problem \eqref{sub-qp} and $\eps\ge 0$ be given. We now present this improved TPCG method for problem \eqref{sub-qp} as follows.

\gap

%\vspace{.1in}

\noindent {\bf Subroutine 2: $y = \TPCG2(A,b,c,J_0,J_-,J_+,x^0,\eps)$}
%\footnote{$x^0$ is a feasible point of problem \eqref{sub-qp} and  $\eps$ is a nonnegative number.}

\vspace{.1in}

%{\bf Input:} $\cH_0$, $c$, $\eps$, and $x^0\in\cH_0$.

\noindent{\bf Input:} $A$, $b$, $c$, $J_0$, $J_-$, $J_+$, $x^0$, $\eps$.

\vspace{.1in}

\noindent Set $\cH_0 = \{x\in\Re^n: x_j = 0, \ j \in J_0\}$, $J^0_0=J_0$, $J^0_-=J_-$, $J^0_+=J_+$, $k=0$.

\vspace{.1in}

\noindent {\bf Repeat}
\bi
\item[1)] If $\|\cP_{\cH_k}(Ax^k-b+c)\|_\infty \le \eps$, return
$y=x^k$ and terminate.
\item[2)] $x^{k+1} = \TPCG1(A,b,c,J^k_0,J^k_-,J^k_+,x^k,\eps)$.
%$x^{k+1} = \TPCG1(\cH_k,c,J^k_0,J^k_-,J^k_+,x^k,\eps)$.
\item[3)] $J^{k+1}_0=I_0(x^{k+1})$, $J^{k+1}_-=I_-(x^{k+1})$, $J^{k+1}_+=I_+(x^{k+1})$, $\cH_{k+1}=\cH(x^{k+1})$.
%and $$\cH_{k+1} = \{x\in\Re^n: x_j=0, \ j\in J^{k+1}_0\}.$$
%\item[3)] If $\|\cP_{\cH_{k+1}}(g(x^{k+1}))\|_\infty \le \eps$, return
%$y=x^{k+1}$ and terminate.
\item[4)] $k \leftarrow k+1$.
\ei
{\bf Output:} $y$.

\gap

We next show that under some suitable assumptions, Subroutine 2 
 terminates in a finite number of iterations.

\begin{theorem} \label{TPCG2-converge}
Assume that problem \eqref{sub-qp} has at least an optimal solution. 
Let $\cH_0 = \{x\in\Re^n: x_j=0, \ j\in J_0\}$ and $\Omega$ be the feasible 
region of problem \eqref{sub-qp}. Suppose that $x^0$ is a feasible point \eqref{sub-qp} and $\eps \ge 0$.  Then the following statements hold:
\bi
\item[(i)] Subroutine 2 is well defined.
\item[(ii)] Subroutine 2 terminates in at most $n+1-|J_0|$ iterations. Moreover, its output $y$ satisfies $y\in\Omega$ and $\|\cP_{\cH(y)}(Ay-b+c)\|_\infty \le \eps$,
where $\cH(\cdot)$ is defined in \eqref{Fxs}.
\item[(iii)] Suppose additionally that $\|\cP_{\cH_0}(Ax^0-b+c)\|_\infty > \eps$ and $x^0-\alpha \cP_{\cH_0}(Ax^0-b+c) \in \Omega$ for sufficiently small $\alpha>0$. Then $q(y) < q(x^0)$, where $q$ is defined in \eqref{sub-qp}.
\ei
\end{theorem}

\begin{proof}
(i) Observe that in step 2) of Subroutine 2, Subroutine 1 (namely, $\TPCG1$)
%\[
%x^{k+1} = \TPCG1(A,b,c,J^k_0,J^k_-,J^k_+,x^k,\eps)
%\]
is applied to the problem
\beq \label{kth-subqp}
\ba{rl}
\min\limits_x & q(x) \\
\mbox{s.t.} & x_j = 0, \quad  j \in J^k_0,\\
& x_j \le 0,  \quad  j \in J^k_-, \\
& x_j \ge 0,  \quad  j \in J^k_+,
\ea
\eeq
where $q$ is defined in \eqref{sub-qp}.  Let $\Omega_k$ denote the feasible region of \eqref{kth-subqp}. In view of the updating scheme of Subroutine 2 and the definitions of $J^k_0$, $J^k_-$ and $J^k_+$,  it is not hard to observe that
$\emptyset \neq \Omega_k \subseteq \Omega$. By the assumption that \eqref{sub-qp} has at least an
optimal solution, so does \eqref{kth-subqp}. It then follows from Theorem \ref{TPCG1-converge} that
$x^{k+1}$ shall be successfully generated by Subroutine 1. Using this observation and an inductive argument, we can conclude that Subroutine 2 is well defined.  
%after running a finite of iterations within this subroutine.

(ii) Suppose for contradiction that Subroutine 2 does not terminate in $K=n+1-|J_0|$ iterations. Then  $\|\cP_{\cH_{k+1}}(Ax^{k+1}-b+c)\|_\infty > \eps$ for 
all $0 \le k \le K$. Since $\{x^k\}_{k=0}^{K}$ are generated by Subroutine 1, one can observe that $I_0(x^k) \subseteq I_0(x^{k+1})$ 
and hence $\cH_k \supseteq \cH_{k+1}$ for every $0\le k \le K$.  It then follows from these and the definition of $\cH(\cdot)$  that for all $0\le k \le K$, 
\[
\|\cP_{\cH_k}(Ax^{k+1}-b+c)\|_\infty \ge \|\cP_{\cH_{k+1}}(Ax^{k+1}-b+c)\|_\infty  > \eps.
\]
This implies that when Subroutine 1 is applied to \eqref{kth-subqp}, it  
terminates at a boundary point $x^{k+1}$ of the feasible region of \eqref{kth-subqp}. It then follows that 
\[
I_0(x^0) \subsetneq  I_0(x^1) \subsetneq \cdots \subsetneq I_0(x^K).
\]
Thus $\{|I_0(x^k)|\}^{K}_{k=0}$ is strictly increasing, which along with $K=n+1-|J_0|$ and $|I_0(x^0)| \ge |J_0|$ leads to  $|I_0(x^{K})| \ge n+1$. This contradicts the trivial fact $|I_0(x^K)| \le n$. 
 Therefore, Subroutine 2 must terminate at some $y$ in at most $n+1-|J_0|$ iterations. Clearly, $y\in\Omega$. We now prove $\|\cP_{\cH(y)}(Ay-b+c)\|_\infty \le \eps$ by considering two separate cases as follows.

Case 1): $\|\cP_{\cH_0}(Ax^0-b+c)\|_\infty \le \eps$. In this case, Subroutine 2 terminates at $k=0$ and outputs $y=x^0$. By $x^0\in\Omega$ and the definition of $\cH(\cdot)$, one can see that $\cH(x^0) \subseteq \cH_0$ and hence
\[
\|\cP_{\cH(x^0)}(Ax^0-b+c)\|_\infty \le \|\cP_{\cH_0}(Ax^0-b+c)\|_\infty \le \eps,
\]
which together with $y=x^0$ implies $\|\cP_{\cH(y)}(Ay-b+c)\|_\infty \le \eps$.

Case 2): $\|\cP_{\cH_0}(Ax^0-b+c)\|_\infty > \eps$. In this case, Subroutine 2 must terminate at some iteration $k \ge 1$. It then follows that $\|\cP_{\cH_k}(Ax^k-b+c)\|_\infty \le \eps$ and $y=x^k$. In addition, we observe from the definitions of $\cH(\cdot)$ and $\cH_k$ that $\cH(x^k)=\cH_k$ for $k \ge 1$. It then immediately follows that $\|\cP_{\cH(y)}(Ay-b+c)\|_\infty \le \eps$.

(iii) We now prove statement (iii). Since $\|\cP_{\cH_0}(Ax^0-b+c)\|_\infty > \eps$, $x^1$ must be generated by calling the subroutine $\TPCG1(A,b,c,J_0,x^0,\eps)$, whose first iteration performs a
projected gradient step to find a point $x(\alpha^*)$, where
\[
\alpha^* = \arg\min\limits_{\alpha \ge 0} \{q(x(\alpha)): x(\alpha) \in \Omega\},
\]
and $x(\alpha) = x^0 - \alpha \cP_{\cH_0}(\nabla q(x^0))=x^0-\alpha \cP_{\cH_0}(Ax^0-b+c)$. By the 
assumption that $x^0-\alpha \cP_{\cH_0}(Ax^0-b+c) \in \Omega$ for sufficiently small $\alpha>0$, one can see that $\alpha^*>0$ and $q(x(\alpha^*))<q(x^0)$. We also observe that the value of $q$ is non-increasing along the subsequent iterates of the subroutine $\TPCG1(A,b,c,J_0,x^0,\eps)$. These observations and the definition of $x^1$ imply that $q(x^1)<q(x^0)$.
 In addition, $q$ is non-increasing along the iterates generated in Subroutine 1. Hence, $q(x^{k+1}) \le q(x^k)$ for all $k \ge 1$. It then follows $q(x^k) < q(x^0)$ for all $k \ge 1$. Notice that $y=x^k$
for some $k \ge 1$. Hence, $q(y) < q(x^0)$.
\end{proof}

\gap

{\bf Remark 3:} As seen from Theorem \ref{TPCG2-converge}, the subroutine $\TPCG1$ is executed in $\TPCG2$ at most (but possibly much less than) $n+1$ times. In view of this and Remark 2, one can see that when $\eps=0$, the number of PCG iterations executed in $\TPCG2$ is at most $(n+1)^2$. On the other hand, when 
$\eps>0$, its number of PCG iterations depends on $\eps$ in $O(\log(1/\eps))$.

\subsection{The first generalized conjugate gradient method for \eqref{l1-qp}}
\label{1st-cg}

%In this subsection we propose a CG type method for solving problem \eqref{l1-qp}. Assuming no
%numerical error occurs, we show that this method finds an optimal solution of \eqref{l1-qp} in a
%finite number of iterations. We also propose an inexact variant of this method and shows it is
%capable of finding an approximate optimal solution of \eqref{l1-qp} even if numerical errors occur.

In this subsection we propose a $\CG$ method for solving problem \eqref{l1-qp}.  We show that this method terminates at an optimal solution of \eqref{l1-qp} in a finite number of iterations. Before proceeding, we introduce some notations that will be used through the next several subsections.

Given any $x\in\Re^n$, we define
\beq \label{index-I0}
\ba{l}
I^0_0(x) = \{i\in I_0(x): 0 \in [\nabla_i f(x)-\tau, \nabla_i f(x)+\tau]\}, \\ [5pt]
I^+_0(x) = \{i\in I_0(x):  \nabla_i f(x)+\tau<0\}, \\ [5pt]
I^-_0(x) = \{i\in I_0(x): \nabla_i f(x)-\tau>0\},
\ea
\eeq
where $I_0(\cdot)$ is given in \eqref{index-I}. 
%One can observe from \eqref{v}  and \eqref{index-I0} that $v_i(x)=0$ if $i\in I^0_0(x)$. 
 Also, we define
$c(\cdot;\tau): \Re^n \to \{-\tau,0,\tau\}^n$ as follows:
\beq \label{cx}
c_i(x;\tau) = \left\{
\ba{ll}
\tau & \ \mbox{if} \ i\in I_+(x) \cup I^+_0(x); \\ [5pt]
0 & \ \mbox{if} \ i \in I^0_0(x); \\ [5pt]
-\tau &  \ \mbox{if} \ i \in I_-(x) \cup I^-_0(x),
\ea
\right. \quad i=1,\dots, n,
\eeq
where $I_-(\cdot)$ and $I_+(\cdot)$ are defined in \eqref{index-I}. 
% changed on Jan 31
%It then follows from \eqref{v}, \eqref{cx} and the definition of $f$ that
%\beq \label{v-gQ}
%v_i(x) = \nabla_i f(x) + c_i(x;\tau) = Ax-b+c(x;\tau).
%\eeq
It then follows from \eqref{v} and \eqref{cx} that
\beq \label{v-gQ}
v_i(x) = \nabla_i f(x) + c_i(x;\tau), \quad \forall i \notin I^0_0(x). 
\eeq
In addition, given  any $y\in\Re^n$, we define
\[
Q(x;y) = f(x) + c(y;\tau)^T x.
\]

The main idea of our $\CG$ method is as follows. Given a current iterate $x^k$, we check
to see whether $v(x^k)=0$ or not. If yes, then $x^k$ is an optimal solution of \eqref{l1-qp}. Otherwise, we find next iterate $x^{k+1}$ by applying Subroutine 2 with initial point $x^k$ and $\eps=0$ 
to the problem
\beq \label{kth-qp}
\ba{rl}
\min\limits_x & Q(x;x^k) \\
\mbox{s.t.} & x_j = 0, \quad  j \in J^k_0,\\
& x_j \le 0,  \quad  j \in J^k_-, \\
& x_j \ge 0,  \quad  j \in J^k_+,
\ea
\eeq
where $J^k_0=I^0_0(x^k)$, $J^k_-=I_-(x^k) \cup I^-_0(x^k)$ and $J^k_+=I_+(x^k) \cup I^+_0(x^k)$.
That is, $x^{k+1}$ is obtained by executing the subroutine
$
\TPCG2(A,b,c^k,J^k_0,J^k_-,J^k_+,x^k,0).
$
As later shown, such $x^{k+1}$ satisfies the following properties: 
\beqa
&& F(x^{k+1}) < F(x^k), \label{F-descent} \\ [5pt] 
&& x^{k+1} \in \Arg\min\{F(x): x_i=0, \ i \in I_0(x^{k+1})\}. \label{subspace-min} \footnotemark   
\eeqa
\footnotetext{By convention, the symbol $\Arg$ stands for the set of the solutions of the associated optimization problem. When this set is known to be a singleton, we use the symbol $\arg$ to stand for it instead.}
By these relations, one can observe that there is no repetition among $\{I_0(x^k): k\ge 1\}$.
Notice that $\{I_0(x): x\in \Re^n\}$ is a finite set. Thus the method must terminate
in a finite number of iterations. Moreover, we will show that it terminates at  an optimal solution of \eqref{l1-qp}. We now present our $\CG$ method as follows.

\gap

\noindent {\bf $\CG$ method 1 for problem \eqref{l1-qp}: $y = \CG1(A,b,\tau,x^0,\eps)$}

\vspace{.1in}

\noindent{\bf Input:} $A$, $b$, $\tau$, $x^0$, $\eps$.

\vspace{.1in}

\noindent Set $k=0$.

\vspace{.1in}

\noindent {\bf Repeat}
\bi
\item[1)] If $\|v(x^k)\|_\infty \le \eps$, return $y=x^k$ and terminate.
\item[2)] $J^k_0 = I^0_0(x^k)$, $J^k_-=I_-(x^k) \cup I^-_0(x^k)$, $J^k_+=I_+(x^k) \cup I^+_0(x^k)$, $c^k = c(x^k;\tau)$.
\item[3)] $x^{k+1} = \TPCG2(A,b,c^k,J^k_0,J^k_-,J^k_+,x^k,0)$.
\item[4)] $k \leftarrow k+1$.
\ei
{\bf Output:} $y$.

\gap

We next show that the above $\CG$ method terminates in a finite number of iterations, and moreover it finds an optimal solution of problem \eqref{l1-qp} when $\eps=0$.

\begin{theorem} \label{CG1-convergence}
 Under Assumption \ref{assump}, the following statements hold:
\bi
\item[(i)]
$\CG$ method 1 terminates in at most $\L(n)$ iterations, where $\L(n)$ is defined in \eqref{L}.
\item[(ii)] Let $y$ be the output of $\CG$ method 1. Then $s \in \partial F(y)$
for some $s$ with $\|s\|_\infty \le \eps$, and moreover, $y$ is an optimal solution of
problem \eqref{l1-qp} when $\eps=0$.
\ei
\end{theorem}

\begin{proof}
(i)  We first show that \eqref{F-descent} and \eqref{subspace-min} hold at iteration 
$k$ at which $\CG$ method 1 has not yet terminated, that is, 
$\|v(x^k)\|_\infty>\eps$.
Let $\Omega_k$ denote the feasible region of problem \eqref{kth-qp}. 
By \eqref{cx} and the definitions of $J^k_0$, $J^k_-$ and $J^k_+$, one can observe that $c_i(x^k;\tau)x_i = \tau|x_i|$ for all $x\in \Omega_k$. 
This along with the definitions of $Q(\cdot;\cdot)$, $f$ and $F$ yields 
\beq \label{FQ}
F(x) = Q(x;x^k), \quad \quad \forall x\in\Omega_k.
\eeq
It then follows from Assumption \ref{assump} that problem \eqref{kth-qp} is bounded below and
hence has at least an optimal solution. By this and Theorem \ref{TPCG2-converge}, $x^{k+1}$ shall
be successfully generated in step 3) by the subroutine $\TPCG2$. We next show that $x^{k+1}$ satisfies \eqref{F-descent} and \eqref{subspace-min}. To this end, let 
\beq \label{cHk}
\bcH = \{x\in\Re^n: x_i=0, \ i \in I^0_0(x^k)\}.
\eeq
% changed on Jan 31
%By \eqref{v-gQ} and the definition of $c^k$, we have
%\beq \label{proj-H0}
%\cP_{\bcH}(Ax^k-b+c^k) = \cP_{\bcH}(\nabla f(x^k)+c(x^k;\tau)) = \cP_{\bcH}(v(x^k)).
%\eeq
%Observe from \eqref{v} and \eqref{index-I0} that $v_i(x^k)=0$ if $i\in I^0_0(x^k)$; $v_i(x^k) < 0$ if $i\in I^+_0(x^k)$; and $v_i(x^k) > 0$ if $i\in I^-_0(x^k)$. By \eqref{cHk}, one can see that $[\cP_{\bcH}(v(x^k))]_i=0$ if $i\in I^0_0(x^k)$ and $[\cP_{\bcH}(v(x^k))]_i=v_i(x^k)$ otherwise. These and \eqref{cHk} imply that $[\cP_{\bcH}(v(x^k))]_i= 0$ if $i\in I^0_0(x^k)$; $[\cP_{\bcH}(v(x^k))]_i< 0$ if $i\in I^+_0(x^k)$; and $[\cP_{\bcH}(v(x^k))]_i> 0$ if $i\in I^-_0(x^k)$. 
By \eqref{v-gQ}, \eqref{cHk} and the definition of $c^k$, we have
\beq \label{proj-H0}
\cP_{\bcH}(Ax^k-b+c^k) = \cP_{\bcH}(\nabla f(x^k)+c(x^k;\tau)) = \cP_{\bcH}(v(x^k)).
\eeq
Observe from \eqref{v} and \eqref{index-I0} that $v_i(x^k) < 0$ if $i\in I^+_0(x^k)$; and $v_i(x^k) > 0$ if $i\in I^-_0(x^k)$. By \eqref{cHk}, one can see that $[\cP_{\bcH}(v(x^k))]_i=0$ if $i\in I^0_0(x^k)$ and $[\cP_{\bcH}(v(x^k))]_i=v_i(x^k)$ otherwise. These imply that $[\cP_{\bcH}(v(x^k))]_i= 0$ if $i\in I^0_0(x^k)$; $[\cP_{\bcH}(v(x^k))]_i< 0$ if $i\in I^+_0(x^k)$; and $[\cP_{\bcH}(v(x^k))]_i> 0$ if $i\in I^-_0(x^k)$. 
In view of these relations, \eqref{proj-H0} and the definition of $\Omega_k$, one can observe that
when $\alpha >0$ is sufficiently small,
\[
x^k - \alpha \cP_{\bcH}(Ax^k-b+c^k) = x^k - \alpha \cP_{\bcH}(v(x^k)) \in \Omega_k.
\]
In addition, we observe that $\|\cP_{\bcH}(v(x^k))\|_\infty = \|v(x^k)\|_\infty$, which along with \eqref{proj-H0} and $\|v(x^k)\|_\infty>\eps$ yields 
$\|\cP_{\bcH}(Ax^k-b+c^k)\|_\infty >\eps$. 
Recall that $x^{k+1}$ is resulted from the subroutine $\TPCG2$ when applied to problem \eqref{kth-qp} starting at $x^k$. It then follows from Theorem \ref{TPCG2-converge} that
$x^{k+1} \in\Omega_k$, $\cP_{\hcH}(Ax^{k+1}-b+c^k)=0$ and $Q(x^{k+1};x^k) < Q(x^k;x^k)$, where
\[
\hcH = \{x\in\Re^n: x_i=0, \ i \in I_0(x^{k+1})\}.
\]
In view of \eqref{FQ}, $Q(x^{k+1};x^k) < Q(x^k;x^k)$ and $x^k, \ x^{k+1}\in\Omega_k$, we see that \eqref{F-descent} holds. 
%In addition, by \eqref{Q}, the definition of $c^k$ and the relation $\cP_{\hcH}(Ax^{k+1}-b+c^k)=0$, we have $\cP_{\hcH}(\nabla Q(x^{k+1};x^k))=0$. 
%This implies
%\beq \label{Qmin}
%x^{k+1} \in \Arg\min\{Q(x;x^k): x\in\hcH\}.
%\eeq
In addition, since $x^{k+1}\in\Omega_k$, the nonzero components of $x^{k+1}$ share the same sign as the corresponding ones of $x^k$. Using this fact, \eqref{cx}    and the definition of $c^k$, one can observe that $c^k=c(x^k;\tau)\in \tau \partial\|x^{k+1}\|_1$, which along with the definition of $F$ implies that  
$ Ax^{k+1}-b+c^k \in \partial F(x^{k+1})$. It then follows from this, 
$\cP_{\hcH}(Ax^{k+1}-b+c^k)=0$ and $x^{k+1} \in \hcH$ that 
%It then follows that there exists a neighborhood $\cN$ of $x^{k+1}$
%such that $\cN \cap \hcH \subseteq \Omega_k$. This together with \eqref{FQ} and \eqref{Qmin} implies
%\[
%x^{k+1} \in \Arg\min\{F(x): x \in \cN \cap \hcH)\}.
%\]
%This relation and the convexity of $F(\cdot)$ and $\hcH$ further imply
\[
x^{k+1} \in \Arg\min\{F(x): x \in \hcH\}.
\]
This relation and the definition of $\hcH$ immediately imply that 
\eqref{subspace-min} holds. 

We are now ready to prove that $\CG$ method 1 terminates in at most $\L(n)$ 
iterations, where $\L(n)$ is defined in \eqref{L}. Suppose for contradiction that it does 
not terminate in $\L(n)$ iterations. Then this method generates $x^{k+1}$ satisfying \eqref{F-descent} and \eqref{subspace-min} for $k=0, \ldots, \L(n)$. It then follows that for $k=1, \ldots, \L(n)$, 
\[
\min\{F(x): x_i=0, \ i \in I_0(x^{k+1})\} \ < \ \min \{F(x): x_i=0, \ i \in I_0(x^{k})\}. 
\] 
This implies $I_0(x^i) \notin \cI^*$ and $I_0(x^i) \not\subseteq I_0(x^j)$ for all $i, j=1,\ldots,\L(n)$ and $j>i$, where $\cI^*$ is defined in \eqref{L}. This contradicts the definition of $\L(n)$. Thus the method must terminate in at most $\L(n)$ iterations.

(ii) Since $y$ is the output of $\CG$ method 1, one has $\|v(y)\|_\infty \le \eps$. We also know that $v(y) \in\partial F(y)$. Hence, statement (ii) holds with $s=v(y)$. Clearly, when $\eps=0$, we have $0\in \partial F(y)$ and thus $y$ is
an optimal solution of \eqref{l1-qp}.
\end{proof}

\sgap

{\bf Remark 4:} In view of Theorem \ref{CG1-convergence} and Remark 3,  one can observe that the number of PCG iterations executed within $\CG$ method 1 is at most $\L(n)(n+1)^2$.  

\subsection{The second generalized conjugate gradient method for \eqref{l1-qp}}
\label{2nd-cg}

The first $\CG$ method proposed in Subsection \ref{1st-cg} enjoys a nice theoretical property, that is, it terminates at an optimal solution of problem \eqref{l1-qp} in a finite number of iterations when its input parameter $\eps$ is set to $0$. Nevertheless, as observed from step 3) of that method, PCG is required to solve some associated optimization problems exactly. This is not an issue from a theoretical perspective 
due to the finite convergence of PCG. It is, however, generally hard to achieve that due to numerical errors. In this subsection we propose a $\CG$ method for \eqref{l1-qp} in which the involved PCG
is only required to find an approximate solution of the associated optimization problems, which makes the method more practical. To proceed, we introduce some notations and state several facts as follows.

Let $v(\cdot)$ and $I_0(\cdot)$ be defined in \eqref{v} and \eqref{index-I}, respectively. We define the projected minimum-norm subgradient $\vp: \Re^n \to \Re^n$ as follows:
\beq \label{vp}
(\vp(x))_i = \left\{
\ba{ll}
v_i(x) & \ \ \  \mbox{if} \ i \in I_0(x), \\
0 & \ \ \ \mbox{otherwise},
\ea
\right.
\quad\quad \forall x \in \Re^n,
\eeq
which is the projection of $v(x)$ onto the subspace $\{y\in\Re^n: y_i=0, \ i\notin I_0(x)\}$.  As later shown, the direction $-\vp(x)$ is a descent direction for $F$ at $x$ when $\vp(x) \neq 0$. We now assume $\vp(x) \neq 0$. The exact line search starting at $x$ along $-\vp(x)$ can be performed by computing
\beq \label{alpha}
\alpha^* = \arg\min\limits_{\alpha\ge 0} F\left(x-\alpha \vp(x)\right)
\eeq
and setting $x^+ = x-\alpha^* \vp(x)$. We can show that $\alpha^*$ has a closed-form expression. Indeed, by virtue of \eqref{v}, \eqref{index-I0} and
\eqref{vp}, one can observe that for all $\alpha \ge 0$,
\[
x_i - \alpha (\vp(x))_i = \left\{
\ba{ll}
\ge 0 & \ \mbox{if} \ i \in I^+_0(x), \\ [5pt]
\le 0 & \ \mbox{if} \ i \in I^-_0(x), \\ [5pt]
%= 0 & \ \mbox{if} \ i \in I^0_0(x), \\ [5pt]
= x_i & \ \mbox{otherwise}.
\ea
\right.
\]
This together with \eqref{cx} implies that
\[
\tau \|x-\alpha \vp(x)\|_1 = c(x; \tau)^T(x-\alpha \vp(x)), \quad\quad \forall \alpha \ge 0.
\]
It follows from this and the definitions of $F$ and $f$ that
\beq \label{F-alpha}
F\left(x-\alpha \vp(x)\right) = f\left(x-\alpha \vp(x)\right) + c(x; \tau)^T(x-\alpha \vp(x)),
\quad\quad \forall \alpha \ge 0.
\eeq
Using this relation and \eqref{alpha}, we obtain that
\beq \label{1st-opt}
(\vp(x))^T [\nabla f(x-\alpha^* \vp(x)) + c(x;\tau)] = 0.
\eeq
%changed on Jan 31
Observe from \eqref{v} and \eqref{vp} that $[\vp(x)]_i=0$ for every $i\in I^0_0(x)$. This together with \eqref{v-gQ}, \eqref{vp},  \eqref{1st-opt} and the definition of $f$ implies that
\[
\ba{lcl}
0 &=& (\vp(x))^T [\nabla f(x) -\alpha^* A \vp(x) + c(x;\tau)] \\ [5pt]
& =&  (\vp(x))^T [v(x) -\alpha^* A \vp(x)]  \ = \ \|\vp(x)\|^2 -\alpha^* (\vp(x))^T A \vp(x).
\ea
\]
It follows from this and the assumption $\vp(x) \neq 0$ that $(\vp(x))^T A \vp(x) \neq 0$ and hence
\beq \label{alphas}
\alpha^* = \frac{\|\vp(x)\|^2}{(\vp(x))^T A \vp(x)} \ \ge \ 0.
\eeq

Given $x^0\in\Re^n$, one knows from Theorem \ref{errbdd-Fx} that there exists some 
$\eta>0$ that may depend on $x^0$ such that  
\beq \label{cg2-errbd}
F(x) - F^*_x \le  \frac{\eta}{2\|A\|} \|[v(x)]_{I^\c_0(x)}\|^2  \ \mbox{for  all} \ x \ \mbox{with} \ F(x) \le F(x^0),
\eeq
where $F^*_x$ is defined in \eqref{Fxs}. Especially, when $A$ is symmetric positive definite, one can see from Theorem  \ref{errbdd-F} that the above $\eta$ can be chosen as $\kappa(A)$, where $\kappa(A)$ is defined in \eqref{kappa}.  In general, the actual value of the above $\eta$ may be unknown. In what follows, we first assume that the $\eta$ associated with \eqref{cg2-errbd} is known. For the case where the $\eta$ is unknown, we can estimate it by executing a try-and-test strategy that will be discussed afterwards. 
%we only know the existence of the above $\eta$, but its actual value may be unknown. 
%In what follows, we start with the case where the $\eta$ is known. For the case where the $\eta$ is unknown, 
% we can estimate it by executing a try-and-test strategy that will be discussed afterwards. 
% For convenience of presentation, we make an assumption regarding $\eta$ as follows.

%\begin{assumption} \label{assump-eta}
% The $\eta$ associated with \eqref{cg2-errbd} is known. \footnote{If such an $\eta$ is unknown, 
% one can estimate it by performing a try-and-test strategy that will be discussed later.} 
%\end{assumption}

We next propose the second $\CG$ method for problem \eqref{l1-qp}.  Unlike the first $\CG$ method that always performs TPCG2 (namely Subroutine 2), each iteration of the method presented below either executes the subroutine TPCG2 or performs the exact line search along the negative projected minimum-norm subgradient of $F$. Following a similar strategy proposed by Dostal 
and Sch\"oberl \cite{Dostal2} for solving a box-constrained convex QP, we 
determine which type of step should be taken by comparing the magnitude of some 
components of the minimum-norm subgradient of $F$ to that of its rest 
components.
%we use the magnitude of some components of the minimum-norm subgradient of $F$ to determine which type of step should be conducted.  
 In particular, given a current iterate $x^k$, if
\beq \label{measure}
\left\|[v(x^k)]_{I_0(x^k)}\right\| > \sqrt{\eta} \left\|[v(x^k)]_{I^\c_0(x^k)}\right\|,
\eeq
where $\eta$ is given in \eqref{cg2-errbd},  it indicates that the zero components of $x^k$ are more far from being optimal compared to its nonzero components.  It is thus plausible to release some of zero components of $x^k$ by minimizing $F$ along the direction $-\vp(x^k)$ to obtain a new iterate $x^{k+1}$, that is,
\[
x^{k+1} = x^k-\alpha_k \vp(x^k),
\]
where $\alpha_k$ is computed by \eqref{alphas} with $x$ replaced by $x^k$.  Analogously, if \eqref{measure} is violated at $x^k$, it is more 
beneficial to improve the nonzero components of $x^k$, which can be made by the subroutine $\TPCG2$ to result in $x^{k+1}$.

Let  $\eta$ be  given in \eqref{cg2-errbd} and $\eps \ge 0$. The second $\CG$ method for problem \eqref{l1-qp} is presented in detail as follows.

\gap

\noindent {\bf $\CG$ method 2 for problem \eqref{l1-qp}: $y = \CG2(A,b,\tau,x^0,\eta,\eps)$}

\vspace{.1in}

\noindent{\bf Input:} $A$, $b$, $\tau$, $x^0$, $\eta$, $\eps$.

\vspace{.1in}

\noindent Set $k=0$.

\vspace{.1in}

\noindent {\bf Repeat}
\bi
\item[1)] If $\|v(x^k)\|_\infty \le \eps$, return $y=x^k$ and terminate.
\item[2)] $J^k_0 = I^0_0(x^k)$, $J^k_-=I_-(x^k) \cup I^-_0(x^k)$, $J^k_+=I_+(x^k) \cup I^+_0(x^k)$, $c^k = c(x^k;\tau)$.
\item[3)] If \eqref{measure} holds,
%$\left\|[v(x^k)]_{I_0(x^k)}\right\| > \kappa(A) \left\|[v(x^k)]_{(I_0(x^k))^\c}\right\|$,
do
 \beq \label{step-21} 
 x^{k+1} = x^k-\alpha_k \vp(x^k), \quad \alpha_k = \frac{\|\vp(x^k)\|^2}{(\vp(x^k))^T A \vp(x^k)};
\eeq 
else
\beq \label{step-22}
x^{k+1} = \TPCG2\left(A,b,c^k,J^k_0,J^k_-,J^k_+,x^k,\frac{\eps}{\max(\sqrt{n\eta},1)}\right).
\eeq
\item[4)] $k \leftarrow k+1$.
\ei
{\bf Output:} $y$.

\gap

Before establishing convergence of $\CG$ method 2, we study two sufficient descent properties of $F$ at $x$ along the direction of $-\vp(x)$.
%along the direction of negative projected minimum-norm subgradient.

\begin{lemma} \label{lem:descent1}
Suppose that $x\in\Re^n$ satisfies $\vp(x) \neq 0$. Let $x^+ = x-\alpha^* \vp(x)$, where
$\alpha^*$ be defined in \eqref{alphas}. Then there holds:
\beq \label{descent}
F(x^+) \ = \ F(x) - \frac{\|\vp(x)\|^4}{2(\vp(x))^T A \vp(x)}
\ \le \ F(x) - \frac{\|\vp(x)\|^2}{2\|A\|}.
\eeq

\end{lemma}

\begin{proof}
In view of \eqref{cx}, one has $c(x; \tau)^T x = \tau \|x\|_1$ and hence
\[
F(x) = f(x) + c(x; \tau)^T x.
\]
Using this, \eqref{alpha}-\eqref{1st-opt} and the definition of $f$, we have
\[
\ba{lcl}
F(x) &=& f(x) + c(x; \tau)^T x  \ = \  f\left(x-\alpha^* \vp(x)\right) + c(x; \tau)^T(x-\alpha^* \vp(x)) \\ [8pt]
& & +\  \alpha^* (\vp(x))^T [\nabla f(x-\alpha^* \vp(x)) + c(x;\tau)] + \frac12 (\alpha^*)^2 (\vp(x))^T A \vp(x)
\\ [8pt]
&=&  F(x-\alpha^* \vp(x)) + \frac12 (\alpha^*)^2 (\vp(x))^T A \vp(x) \\ [8pt]
&=& F(x^+) + \frac{\|\vp(x)\|^4}{2(\vp(x))^T A \vp(x)} \ \ge \ F(x^+) + \frac{\|\vp(x)\|^2}{2\|A\|},
\ea
\]
where the last inequality follows from $z^T A z \le \|A\| \|z\|^2$ for all $z\in\Re^n$.
\end{proof}

\gap

\begin{lemma} \label{lem:descent2}
Let $\eta$ be given in \eqref{cg2-errbd}. Suppose that $x\in\Re^n$ satisfies $F(x) \le F(x^0)$ and 
\beq \label{measure-x}
\left\|[v(x)]_{I_0(x)}\right\| > \sqrt{\eta} \left\|[v(x)]_{I^\c_0(x)}\right\|.
\eeq
Let $F^*_x$ be defined in \eqref{Fxs} and $x^+ = x-\alpha^* \vp(x)$, where $\alpha^*$ is defined in \eqref{alphas}.
Then $F(x^+) < F^*_x$. 
%\[
%F(x^+) < F^*_x := \min\{F(z): z_i =0, \ i \in I_0(x)\}.
%\]
\end{lemma}

\begin{proof}
Let $J = I^\c_0(x)$. 
%By the definition of $F^*_x$ and \eqref{cg2-errbd}, one has
%\[
%F(x) - F^*_x  \ \le \ \frac{\eta}{2} \|A^+\|  \|(v(x))_J\|^2,
%\]
It follows from  \eqref{descent} and \eqref{cg2-errbd} that
\[
\ba{lcl}
F(x^+) & \le &  F(x) - \frac{\|\vp(x)\|^2}{2\|A\|} \ \le \  F^*_x + \frac{\eta}{2\|A\|} \|(v(x))_J\|^2 - \frac{\|\vp(x)\|^2}{2\|A\|}, \\ [10pt]
&=& F^*_x + \frac{\|\vp(x)\|^2}{2\|A\|} \left(\frac{\eta\|(v(x))_J\|^2}{\|\vp(x)\|^2}-1\right).
\ea
\]
The conclusion follows from this inequality, $\|\vp(x)\|=\|[v(x)]_{I_0(x)}\|$ and \eqref{measure-x}.
\end{proof}

\gap

We next show that $\CG$ method 2 terminates in a finite number of iterations, and moreover it terminates at an optimal solution of problem \eqref{l1-qp} when $\eps=0$.

\begin{theorem} \label{CG2-convergence}
Under Assumption \ref{assump}, the following statements hold:
\bi
\item[(i)]
 $\CG$ method 2 terminates in at most $2\L(n)$ iterations, where $\L(n)$ is defined in \eqref{L}. 
\item[(ii)] Let $y$ be the output of $\CG$ method 2. Then $s \in \partial F(y)$
for some $s$ with $\|s\|_\infty \le \eps$, and moreover, $y$ is an optimal solution of
problem \eqref{l1-qp} if $\eps=0$.
\ei
\end{theorem}

\begin{proof}
(i) We first claim that $F(x^{k+1}) \le F(x^k)$ for all $k \ge 0$. Indeed, $x^{k+1}$ is generated by \eqref{step-21} or \eqref{step-22}. 
If it is generated by \eqref{step-21}, it follows from \eqref{descent} that $F(x^{k+1}) \le F(x^k)$. On  the other hand , if $x^{k+1}$ is generated by \eqref{step-22}, one can observe from the subroutine $\TPCG2$  that $F(x^{k+1}) \le F(x^k)$.

Secondly, we claim that the number of executions of \eqref{step-21} is at most $\L(n)$.
Indeed, let 
\[
\cK=\{k: x^{k+1} \ \mbox{is generated by} \ \eqref{step-21}\}.
\]
By the updating scheme of $\CG$ method 2, it can be observed that \eqref{measure} 
holds at $x^k$ for all $k\in\cK$. In addition, by the monotonicity of $\{F(x^k)\}$, it is 
clear that $F(x^k) \le F(x^0)$ for every $k\in\cK$. It then follows from Lemma 
\ref{lem:descent2} that $F(x^{k+1}) < F^*_{x^k}$ for all $k\in\cK$. By \eqref{Fxs}, we 
know that $F(x^{k+1}) \ge F^*_{x^{k+1}}$. In view of these relations and  the monotonicity of
%% changed on Feb 1
 $\{F(x^k)\}$, one has  $F^*_{x^{j+1}} \le F(x^j) \le F(x^{i+1}) < F^*_{x^{i+1}}$ for all $i,j \in \cK$ and 
$j>i$. It follows from this relation and \eqref{Fxs} that $I_0(x^{i+1}) \notin \cI^*$ and $I_0(x^{i+1}) \not\subseteq I_0(x^{j+1})$ for all $i,j \in \cK$ and $j>i$, where $\cI^*$ is defined in \eqref{L}. By this and \eqref{L}, we see that $|\cK| \le \L(n)$ and thus this claim holds.

Thirdly, we claim that \eqref{step-22} cannot be executed at any two consecutive iterations. Suppose for contradiction that \eqref{step-22} is executed at iterations $k$ and $k+1$ for some $k\ge 0$. By the updating scheme of $\CG$ method 2, we then know that \eqref{measure} does not hold at $x^k$ and $x^{k+1}$. It follows from \eqref{step-22} and Theorem \ref{TPCG2-converge}  that
\beq \label{approx-opt}
\|\cP_{\cH(x^{k+1})}(Ax^{k+1}-b+c^k)\|_\infty \le \frac{\eps}{\max(\sqrt{n\eta},1)},
\eeq
where $\cH(\cdot)$ is defined in \eqref{Fxs}. Observe that all nonzero components of $x^{k+1}$ share the same sign as the corresponding ones of $x^k$. This together with \eqref{cx} implies that $c_i(x^{k+1};\tau) = c_i(x^k;\tau)=c^k_i$ for all $i\in I^\c_0(x^{k+1})$. Using this fact, \eqref{v-gQ} and the definition of $\cH(\cdot)$, we have
\[
\|\cP_{\cH(x^{k+1})}(Ax^{k+1}-b+c^k)\|_\infty=\|\cP_{\cH(x^{k+1})}(Ax^{k+1}-b+c^{k+1})\|_\infty =  \|\cP_{\cH(x^{k+1})}(v(x^{k+1}))\|_\infty.
\]
It follows from this, \eqref{approx-opt} and  the definition of $\cH(\cdot)$ that 
\beq \label{vI-20}
\left\|[v(x^{k+1})]_{I^\c_0(x^{k+1})}\right\|_\infty \le \frac{\eps}{\max(\sqrt{n\eta},1)} \ \le \ \eps.
\eeq
In view of this and the fact that \eqref{measure} does not hold at $k+1$, one has
\[
\ba{lcl}
\left\|[v(x^{k+1})]_{I_0(x^{k+1})}\right\|_\infty  &\le &  \left\|[v(x^{k+1})]_{I_0(x^{k+1})}\right\|
\ \le \ \sqrt{\eta} \left\|[v(x^{k+1})]_{\left(I_0(x^{k+1})\right)^\c}\right\|,  \\ [8pt]
& \le & \sqrt{n\eta} \left\|[v(x^{k+1})]_{I^\c_0(x^{k+1})}\right\|_\infty \ \le \  \frac{\eps\sqrt{n\eta}}{\max(\sqrt{n\eta},1)} \ \le \ \eps.
\ea
\]
This together with \eqref{vI-20} yields $\|v(x^{K+1})\|_\infty \le \eps$. Hence, $\CG$ method 2 terminates at $x^{k+1}$ and thus \eqref{step-22} will not be executed at iteration $k+1$, which contradicts the above supposition. Therefore, this claim holds. 

From the last claim above, one can see that \eqref{step-22} is executed at most  
once between every two adjacent executions of \eqref{step-21}. In view of this fact, the second claim above and the updating scheme of $\CG$ method 2, 
 we can conclude that it must terminate in at most $2\L(n)$ iterations.  

(ii) The proof of the second statement is similar to that of Theorem \ref{CG1-convergence}.
\end{proof}

\gap

{\bf Remark 5:} From the proof of Theorem \ref{CG2-convergence}, we know that the subroutine $\TPCG2$ is called  in $\CG2$ at most $\L(n)$ times. In view of this and Remark 3,  one can observe that when $\eps=0$, the number of PCG iterations executed within $\CG$ method 2 is at most $\L(n)(n+1)^2$. On the other hand, when $\eps>0$,  its number of PCG iterations depends on $\eps$ in $O(\log(1/\eps))$.

\vspace{.1in}
%To run the above CG type method 2, the $\eta$ satisfying \eqref{cg2-errbd}  needs to be known 
%beforehand. In general, we only know the existence of such $\eta$, but its actual value is 
%unknown. We next develop a variant of CG type method 2 that does not use the actual value of 
%such $\eta$.  

The above $\CG$ method is suitable for the case where the $\eta$ associated with \eqref{cg2-errbd} is known. From the proof of Theorem \ref{CG2-convergence}, it
can be observed that the error bound \eqref{cg2-errbd} with a known $\eta$ ensures 
that $I_0(x^i)$ is not a subset of $I_0(x^j)$ for all $I^0(x^i), I^0(x^j)\in\cC_\eta$ with 
$j>i$, where 
\beq \label{ceta}
\cC_\eta=\left\{I_0(x^k): x^k  \ \mbox{satisfies} \ \eqref{measure} \ \mbox{with the given} \  \eta.\right\}.
\eeq
When such $\eta$ is unknown, we can start with a guess of $\eta$, denoted by $\heta$. Then we run the subroutine $\CG$2 with $\eta$ replaced by $\heta$ until some $x^k$ is found with the property: (a) $\|v(x^k)\|_\infty \le \eps$ or (b) $x^k$ satisfies \eqref{measure} with $\eta$ replaced by $\heta$ and $I_0(x^k)$ is a supset of some member in $\cC_\heta$ that was previously generated, where $\cC_\heta$ is defined according to \eqref{ceta} by replacing $\eta$ by $\heta$. If (a) occurs, $x^k$ is a desired approximate solution of problem \eqref{l1-qp}. On  the other hand ,  if (b) occurs, it follows from the above observation that $\heta$ is clearly a wrong guess of $\eta$, and we need to increase $\heta$ and repeat the above process starting with $x^k$. These observations lead to the following variant of $\CG$ method 2.
%, which is suitable for the case where the $\eta$ in \eqref{cg2-errbd} is unknown. 
%and terminates in a finite number of iterations.
% at an approximate solution of \eqref{l1-qp}.    
%The main idea of this variant is as follows. We start with a guess of $\eta$, denoted by $\heta$. As seen from Theorem \ref{errbdd-F}, when $A$ is positive definite, the actual value of $\eta$ is 1. Therefore, we choose $\heta=1$ initially.  
%As shown later on, this process will terminate in a finite number of    

\gap

\noindent {\bf A variant of $\CG$ method 2 for problem \eqref{l1-qp}: $y = \CG2_\v(A,b,\tau,x^0,\eta_0,\rho,\eps)$}

\vspace{.1in}

\noindent{\bf Input:} $A$, $b$, $\tau$, $x^0$, $\eta_0$, $\rho>1$, $\eps$.

\vspace{.1in}

\noindent Set $k=0$, $\heta=\eta_0$, $\cC=\emptyset$. 

\vspace{.1in}

\noindent {\bf Repeat}
\bi
\item[1)] If $\|v(x^k)\|_\infty \le \eps$, return $y=x^k$ and terminate.
\item[2)] $J^k_0 = I^0_0(x^k)$, $J^k_-=I_-(x^k) \cup I^-_0(x^k)$, $J^k_+=I_+(x^k) \cup I^+_0(x^k)$, $c^k = c(x^k;\tau)$.
\item[3)] If \eqref{measure} holds with $\eta=\heta$,
%$\left\|[v(x^k)]_{I_0(x^k)}\right\| > \kappa(A) \left\|[v(x^k)]_{(I_0(x^k))^\c}\right\|$,
do 

\quad\quad\quad if $\cC \neq \emptyset$ and some member in $\cC$ is a subset of $I_0(x^k)$, set 
\beq \label{step-2v1}
\heta \leftarrow \rho \heta, \ \cC \leftarrow \emptyset, \ x^{k+1}=x^k;
\eeq
\quad\quad\quad else
\beqa
&& \cC \leftarrow \cC \cup \{I_0(x^k)\}, \nn\\
&& x^{k+1} = x^k-\alpha_k \vp(x^k), \quad \alpha_k = \frac{\|\vp(x^k)\|^2}{(\vp(x^k))^T A \vp(x^k)}; \label{step-2v2} 
\eeqa
else
\beq \label{step-2v3}
x^{k+1} = \TPCG2\left(A,b,c^k,J^k_0,J^k_-,J^k_+,x^k,\frac{\eps}{\max(\sqrt{n\heta},1)}\right).
\eeq 
\item[4)] $k \leftarrow k+1$.
\ei
{\bf Output:} $y$.

\gap

We now briefly discuss how to choose $\eta_0$ for the above method. 
When the $\eta$ associated with \eqref{cg2-errbd} is known, we simply choose $\eta_0=\eta$. It can be observed from the proof below that the variant of $\CG$ method 2 with such a choice of $\eta_0$ is identical to $\CG$ method 2. In addition, when $A$ is symmetric positive definite, we see from  Theorem \ref{errbdd-F} that $\eta=\kappa(A)$ satisfies the error bound \eqref{cg2-errbd}. Thus, it is reasonable to choose $\eta_0=\kappa(A)$ when 
the $\eta$ associated with \eqref{cg2-errbd} is unknown.   

We next show that the variant of $\CG$ method 2 terminates in a finite number of iterations, and moreover it terminates at an optimal solution of problem \eqref{l1-qp} when $\eps=0$.

\begin{theorem} \label{CG2v-convergence}
Under Assumption \ref{assump}, the following statements hold:
\bi
\item[(i)]
 The variant of $\CG$ method 2 terminates in at most 
\beq \label{N}
N = \max\left(\left\lceil \frac{\log \eta^* - \log \eta_0}{\log \rho}\right\rceil+1,1\right)(2^{n+1}-1)
\eeq
 iterations, where $\eta^*$ is the smallest $\eta$ satisfying \eqref{cg2-errbd}.
\item[(ii)] Let $y$ be the output of the variant of $\CG$ method 2. Then $s \in \partial F(y)$
for some $s$ with $\|s\|_\infty \le \eps$, and moreover, $y$ is an optimal solution of
problem \eqref{l1-qp} if $\eps=0$.
\ei
\end{theorem}

\begin{proof}
(i) Let $\eta^*$ be the smallest $\eta$ satisfying \eqref{cg2-errbd}. By the monotonicity of $\{F(x^k)\}$ and a similar argument as in the proof of Theorem \ref{CG2-convergence},  one can show that if $\heta \ge \eta^*$ at some iteration $K$, then $I_0(x^i) \not\subseteq I_0(x^k)$ for all $I_0(x^i) \in \cC$ with $i<k$ and every $k \ge K$ such that \eqref{measure} holds at $x^k$ with $\eta=\heta$. 
Thus $\heta$ will no longer be updated for all $k \ge K$. By this fact and the updating scheme of the variant of $\CG$ method 2,  it is not hard to show that $\heta$ can be updated by \eqref{step-2v1} in at most $\max(\lceil (\log \eta^* - \log \eta_0)/\log \rho\rceil,0)$ times. Hence, the number of distinct $\heta$ 
arising in this method is at most  $\max(\lceil (\log \eta^* - \log \eta_0)/\log \rho \rceil,0)+1$. Observe that if \eqref{measure} holds at some $x^k$, then $I_0(x^k) \neq \emptyset$.  In view of this and the updating scheme of $\cC$, one can see that if $\cC \neq \emptyset$, then 
 all members of $\cC$ are distinct nonempty subsets of $\{1,\ldots,n\}$.  It follows that   
for each  $\heta$, the number of members of $\cC$ is at most $2^n-1$ and hence 
the number of executions of \eqref{step-2v2} is at most $2^n-1$.  By this fact and a similar argument as in the proof of Theorem \ref{CG2-convergence}, one can show that for each $\heta$, the number of executions of \eqref{step-2v3} is also at most $2^n-1$. Thus for each $\heta$, the total 
number of executions of \eqref{step-2v1}, \eqref{step-2v2} and \eqref{step-2v3} is at most 
$1+2(2^n-1)=2^{n+1}-1$. The conclusion of the first statement immediately follows from these facts.

(ii) The proof of the second statement is similar to that of Theorem \ref{CG1-convergence}.
\end{proof}

\vspace{.05in}

{\bf Remark 6:} From the proof of Theorem \ref{CG2v-convergence}, we know that the subroutine $\TPCG2$ is called in $\CG2_v$ at most $\max(\lceil (\log \eta^* - \log \eta_0)/\log \rho \rceil+1,1)(2^n-1)$ times. In view of this and Remark 3,  one can observe that when $\eps=0$, the number of PCG iterations executed in $\CG2_v$ is at most $\max(\lceil (\log \eta^* - \log \eta_0)/\log \rho \rceil+1,1)(2^n-1)(n+1)^2$. On the other hand, when $\eps>0$,  the number of PCG iterations executed in this method depends on $\eps$ in $O(\log(1/\eps))$.

\subsection{The third generalized conjugate gradient method for \eqref{l1-qp}}
\label{3rd-cg}

Notice that the subroutine $\TPCG2$ is used in the variant of $\CG$ method 2. 
Given that $\TPCG2$ is generally more expensive than the subroutine $\TPCG1$, 
a natural question is whether one could replace $\TPCG2$ by $\TPCG1$ there. 
In what follows, we propose a third $\CG$ method by performing such a replacement  and also modifying the associated $J^k_0$, $J^k_-$ and $J^k_+$ accordingly. 

%In this subsection, we propose a third
%$\CG$ method, in which each iteration either performs an exact line search along the negative projected minimum-norm sugradient direction over the boundary of a quadrant or finds a point by executing Subroutine 1 that is usually
%cheaper than Subroutine 2. In particular, given a current iterate $x^k$, if \eqref{measure} holds,
%we obtain $x^{k+1}$ by minimizing $F$ along the projected minimum-norm subgradient $\vp(x^k)$ in the same
%manner as CG type method 2. Otherwise, we perform Subroutine 1 to obtain $x^{k+1}$.
%
%The third CG type method for problem \eqref{l1-qp} is presented in detail as follows.

%\vspace{.1in}
\gap

\noindent {\bf $\CG$ method 3 for problem \eqref{l1-qp}: $y = \CG3(A,b,\tau,x^0,\eta_0,\rho,\eps)$}

\vspace{.04in}

\noindent{\bf Input:} $A$, $b$, $\tau$, $x^0$,  $\eta_0$, $\rho>1$,  $\eps$.

\vspace{.04in}

\noindent Set $k=0$, $\heta=\eta_0$, $\cC=\emptyset$. 

\vspace{.04in}

\noindent {\bf Repeat}
\bi
\item[1)] If $\|v(x^k)\|_\infty \le \eps$, return $y=x^k$ and terminate.
\item[2)] $J^k_0 = I_0(x^k)$, $J^k_-=I_-(x^k)$, $J^k_+=I_+(x^k)$,
$c^k = c(x^k;\tau)$.
\item[3)] If \eqref{measure} holds with $\eta=\heta$,
%$\left\|[v(x^k)]_{I_0(x^k)}\right\| > \kappa(A) \left\|[v(x^k)]_{(I_0(x^k))^\c}\right\|$,
do 

\quad\quad\quad if $\cC \neq \emptyset$ and some member in $\cC$ is a subset of $I_0(x^k)$, set 
\beq \label{step-31}
\heta \leftarrow \rho \heta, \ \cC \leftarrow \emptyset, \ x^{k+1}=x^k;
\eeq
\quad\quad\quad else
\beqa
&& \cC \leftarrow \cC \cup \{I_0(x^k)\}, \nn \\
&& x^{k+1} = x^k-\alpha_k \vp(x^k), \quad \alpha_k = \frac{\|\vp(x^k)\|^2}{(\vp(x^k))^T A \vp(x^k)}; \label{step-33}
\eeqa
else
\beq
 x^{k+1} = \TPCG1\left(A,b,c^k,J^k_0,J^k_-,J^k_+,x^k,\frac{\eps}{\max(\sqrt{n\heta},1)}\right)
 \label{subprob3}.
\eeq
%\item[3)] If \eqref{measure} holds,
%do
%\beq
% x^{k+1} = x^k-\alpha_k \vp(x^k), \quad \alpha_k = \frac{\|\vp(x^k)\|^2}{(\vp(x^k))^T A \vp(x^k)}; \label{ext-grad1}
%\eeq
%else
%\beq
% x^{k+1} = \TPCG1\left(A,b,c^k,J^k_0,J^k_-,J^k_+,x^k,\frac{\eps}{\sqrt{n\eta}}\right)
% \label{subprob3}.
%\eeq
\item[4)] $k \leftarrow k+1$.
\ei
{\bf Output:} $y$.

\vspace{.1in}

{\bf Remark 7:} The parameter $\eta_0$ for this method can be chosen similarly as for the variant of $\CG$ method 2. In particular, when the $\eta$ associated with \eqref{cg2-errbd} is known, we choose $\eta_0=\eta$. Otherwise, we can choose $\eta_0=\kappa(A)$. 

\vspace{.1in}

We next show that $\CG$ method 3 terminates in a finite number of iterations, and moreover it terminates at an optimal solution of problem \eqref{l1-qp} when $\eps=0$.

\begin{theorem} \label{CG3-convergence}
Under Assumption \ref{assump}, the following statements hold:
\bi
\item[(i)] $\CG$ method 3 terminates in at most 
\[
\max\left(\left\lceil \frac{\log \eta^* - \log \eta_0}{\log \rho}\right\rceil+1,1\right)(1+(n+2)(2^n-1))
\]
 iterations, where $\eta^*$ is the smallest $\eta$ satisfying \eqref{cg2-errbd}.
\item[(ii)] Let $y$ be the output of $\CG$ method 3. Then $s \in \partial F(y)$
for some $s$ with $\|s\|_\infty \le \eps$, and moreover, $y$ is an optimal solution of
problem \eqref{l1-qp} if $\eps=0$.
\ei
\end{theorem}

\begin{proof}
(i) We first claim that $F(x^{k+1}) \le F(x^k)$ for all $k\ge 0$. Indeed, if 
$x^{k+1}$ is generated by \eqref{step-31}, $F(x^{k+1})=F(x^k)$. In 
addition, if $x^{k+1}$ is updated by \eqref{step-33}, it follows from Lemma \ref{lem:descent1} that $F(x^{k+1}) \le F(x^k)$. On  the other hand , if $x^{k+1}$ is generated by \eqref{subprob3}, one can observe from the subroutine $\TPCG1$ that 
$F(x^{k+1}) \le F(x^k)$. 

Secondly, we claim that \eqref{subprob3} cannot be executed at any $n+2$ 
consecutive iterations. Suppose for contradiction that \eqref{subprob3} is executed at iterations $i, i+1, \ldots, i+n+1$ iterations for some $i\ge 0$. This along with the updating scheme of $\CG$ method 3 implies that \eqref{measure} with $\eta=\heta$ does not hold at these $n+2$ iterations. We now show that the subroutine $\TPCG1$ executed in \eqref{subprob3}  must terminate at a boundary point $x^{k+1}$ 
of the feasible region of problem \eqref{sub-qp} with $c=c^k$, $J_0=J^k_0$, $J_-=J^k_-$ and $J_+=J^k_+$ for all $i \le k \le i+n$. Suppose not. By the termination 
criteria of $\TPCG1$, there exists some $i \le k\le i+n$ such that 
\beq \label{x-opt}
\|\cP_\cH(Ax^{k+1}-b+c^k)\|_\infty \le \frac{\eps}{\max(\sqrt{n\heta},1)},
\eeq
where
\[
\cH = \{x\in\Re^n: x_j = 0, \ j \in J^k_0\}.
\]
Observe from the subroutine $\TPCG1$ that for such $k$, $J^k_0 = I_0(x^k) \subseteq I_0(x^{k+1})$ and moreover the nonzero components of $x^{k+1}$ share the same sign as the corresponding ones of $x^k$, which implies 
\[
%\ba{l}
%J^k_0 = I_0(x^k) \subseteq I_0(x^{k+1}), \\ [5pt]
c_i(x^{k+1};\tau)=c_i(x^k;\tau) = c^k_i, \quad\forall i \notin I_0(x^{k+1}).
%\ea
\]
Using these relations and \eqref{x-opt}, one can have
\[
\|\cP_\hcH(Ax^{k+1}-b+c(x^{k+1};\tau))\|_\infty \le \frac{\eps}{\max(\sqrt{n\heta},1)},
\]
where
\[
\hcH = \{x\in\Re^n: x_i = 0, \ i \in I_0(x^{k+1})\}.
\]
It follows from this and \eqref{v} that
\[
\left\|[v(x^{k+1})]_{I^\c_0(x^{k+1})}\right\|_\infty \le \frac{\eps}{\max(\sqrt{n\heta},1)} \le \eps.
\]
In view of this,  the fact that \eqref{measure} with $\eta=\heta$ 
does not hold at $x^{k+1}$, and a similar argument as in the proof of 
Theorem \ref{CG2-convergence}, one can show that  
$\|[v(x^{k+1})]_{I_0(x^{k+1})}\|_\infty  \le  \eps$. It then follows that  
$\|v(x^{k+1})\|_\infty \le \eps$, which implies that $\CG$ method 3 
terminates at $x^{k+1}$ and thus \eqref{subprob3} will no longer be executed 
at iteration $k+1$. This contradicts the second supposition above. 
Therefore, TPCG1 must terminate at a boundary point $x^{k+1}$ 
of the feasible region of problem \eqref{sub-qp} with $c=c^k$, $J_0=J^k_0$, $J_-=J^k_-$ and $J_+=J^k_+$ for all $i \le k \le i+n$. This together with the definition of $J^k_0$, $J^k_-$ and $J^k_+$ implies that $I_0(x^k) \subsetneq I_0(x^{k+1})$ and 
hence $|I_0(x^k)| < |I_0(x^{k+1})|$ for all $i \le k \le i+n$, which leads to $|I_0(x^{i+n+1})| \ge n+1$ and contradicts the trivial fact $|I_0(x^{i+n+1})| \le n$. 
Thus the above claim holds. 

From the second claim above, we can see that for each $\heta$, \eqref{subprob3} is executed at most $n+1$ times between every two adjacent executions of \eqref{step-33}. In addition, by the monotonicity of $\{F(x^k)\}$ and a similar argument as in the proof of Theorem \ref{CG2-convergence},  one can show that for each $\heta$, the number of executions of \eqref{step-33} is at most $2^n-1$. Therefore, for each $\heta$ the number of executions of  \eqref{subprob3} is at most $(n+1)(2^n-1)$. 
It follows that for each $\heta$, the total number of executions of \eqref{step-31}, \eqref{step-33} and 
\eqref{subprob3} is at most $1+(n+2)(2^n-1)$. Also, by a similar argument as in the 
proof of Theorem \ref{CG2v-convergence},  we know that  the number of 
distinct $\heta$ arising in this method is at most  $\max(\lceil (\log \eta^* - \log \eta_0)/\log \rho \rceil,0)+1$. The conclusion of the first statement immediately follows from these facts.

(ii) The proof of the second statement is similar to that of Theorem \ref{CG1-convergence}.
\end{proof}

\gap

{\bf Remark 8:} From the proof of Theorem \ref{CG3-convergence}, one knows that the subroutine $\TPCG1$ is called in $\CG3$ at most $(n+1)\max(\lceil (\log \eta^* - \log \eta_0)/\log \rho \rceil+1,1)(2^n-1)$ times. In view of this and Remark 2,  we can observe that when $\eps=0$, the number of PCG iterations executed in $\CG3$ is at most $\max(\lceil (\log \eta^* - \log \eta_0)/\log \rho \rceil+1,1)(2^n-1)(n+1)^2$. On the other hand, when $\eps>0$,  its number of PCG iterations  depends on $\eps$ in $O(\log(1/\eps))$.

\subsection{The fourth generalized conjugate gradient method for \eqref{l1-qp}}
\label{4th-cg}

In this subsection we propose the fourth $\CG$ method for problem \eqref{l1-qp}, which enhances $\CG$ method 3 by incorporating a  
proximal gradient scheme. In particular, we perform a proximal gradient step over a subspace immediately after executing the subroutine  $\TPCG1$, which makes the iterates cross orthants more rapidly and 
also shrinks some nonzero components of the iterates to zero. This 
strategy is similar to the one proposed by Dostal and Sch\"oberl \cite{Dostal2} for solving a box-constrained convex QP.

%The fourth CG type method for problem \eqref{l1-qp} is presented in detail as follows.

\gap

\noindent {\bf $\CG$ method 4 for problem \eqref{l1-qp}: $y = \CG4(A,b,\tau,x^0,t,\eta_0,\rho,\xi,\eps)$}

\vspace{.1in}

\noindent{\bf Input:} $A$, $b$, $\tau$,  $x^0$, $t\in (0,2/\|A\|)$, $\eta_0$, $\rho>1$, $0<\xi<1$, $\eps$. 

\vspace{.1in}

\noindent Set $k=0$, $\heta=\eta_0$, $\heps=\eps$, $\cC=\emptyset$. 

\vspace{.1in}

\noindent {\bf Repeat}
\bi
\item[1)] If $\|v(x^k)\|_\infty \le \eps$, return $y=x^k$ and terminate.
\item[2)] $J^k_0 = I_0(x^k)$, $J^k_-=I_-(x^k)$, $J^k_+=I_+(x^k)$,
$c^k = c(x^k;\tau)$.
\item[3)] If \eqref{measure} holds with $\eta=\heta$,
do 

\quad\quad\quad set $\heps \leftarrow \eps$;

\quad\quad\quad if $\cC \neq \emptyset$ and some member in $\cC$ is a subset of $I_0(x^k)$, set 
\beq \label{step-41}
\heta \leftarrow \rho \heta, \ \cC \leftarrow \emptyset, \ x^{k+1}=x^k;
\eeq
\quad\quad\quad else
\beqa
&& \cC \leftarrow \cC \cup \{I_0(x^k)\}, \label{step-42} \\
&& x^{k+1} = x^k-\alpha_k \vp(x^k), \quad \alpha_k = \frac{\|\vp(x^k)\|^2}{(\vp(x^k))^T A \vp(x^k)}; \label{step-43}
\eeqa
else
\beqa
&& y^{k+1} = \TPCG1\left(A,b,c^k,J^k_0,J^k_-,J^k_+,x^k,\frac{\sqrt{(2t^{-1}-\|A\|)\|A\|}\ \heps}{(t^{-1}+\|A\|)\sqrt{n\heta}}\right), \label{step-44} \\
&& x^{k+1} = \arg\min\left\{\frac12\|x-(y^{k+1}-t(A y^{k+1}-b))\|^2+t \tau\|x\|_1: x \in \cH(y^{k+1})\right\},  
%x_i=0, \ i \in I_0(y^{k+1})\right\}.
\label{step-45} \\
&& \heps \leftarrow \xi \heps. \nn 
\eeqa
\item[4)] $k \leftarrow k+1$.
\ei
{\bf Output:} $y$.

\gap

{\bf Remark:} (a) When the $\eta$ associated with \eqref{cg2-errbd} is known, one can  choose $\eta_0=\eta$. 
%and $\xi=1$. By a similar argument as in the proof of Theorem \ref{CG2v-convergence}, we can show that for such a choice of $\eta_0$, $\heta$ will stay unchanged throughout $\CG$ method 4. It can also be observed from this and the proof below that $\CG$ method 4 with such a choice of parameters terminates within a finite number of iterations. 
Otherwise, one can choose $\eta_0=\kappa(A)$.

(b) It is not hard to see that subproblem \eqref{step-45} has a closed-form solution, which is given as follows:
\[
x^{k+1}_i = \left\{
\ba{ll}
\sgn(a_i) \max(|a_i|-t\tau,0)  & \mbox{if} \ i \notin I_0(y^{k+1}), \\ [5pt]
0 & \mbox{otherwise},
\ea\right. \quad i = 1,\ldots, n,
\]
where $a_i=y^{k+1}_i-t((A y^{k+1})_i-b_i)$ for $i = 1,\ldots, n$.

\gap

We next show that $\CG$ method 4 terminates in a finite number of iterations, and moreover it terminates at an optimal solution of problem \eqref{l1-qp} when $\eps=0$.

\begin{theorem} \label{CG4-convergence}
Under Assumption \ref{assump}, the following statements hold:
\bi
\item[(i)]
 $\CG$ method 4 terminates in at most 
\[
\max\left(\left\lceil \frac{\log \eta^* - \log \eta_0}{\log \rho}\right\rceil+1,1\right)(1+(M+1)(2^n-1)),
\]
 iterations, where $\eta^*$ is the smallest $\eta$ satisfying \eqref{cg2-errbd} and 
\beq \label{M}
M = \max\left(\left\lceil \frac{\max(\log \eta_0 - \log \eta^*,-\log\eta^*)}{2\log \xi}\right\rceil+n+1,n+1\right).
\eeq
\item[(ii)] Let $y$ be the output of $\CG$ method 4. Then $s \in \partial F(y)$
for some $s$ with $\|s\|_\infty \le \eps$, and moreover, $y$ is an optimal solution of
problem \eqref{l1-qp} if $\eps=0$.
\ei
\end{theorem}

\begin{proof}
(i) We first claim that $F(x^{k+1}) \le F(x^k)$ for all $k \ge 0$. By the same argument as in the proof of Theorem \ref{CG3-convergence}, one 
can see that this claim holds if $x^{k+1}$ is generated by \eqref{step-41} or \eqref{step-42}.  We now show that it also holds if $x^{k+1}$ is obtained by \eqref{step-45}. Indeed, by the definition of $f$,  it is not hard to observe that 
\beq \label{lip-ineq}
f(x^{k+1}) \le f(y^{k+1}) + \langle \nabla f(y^{k+1}), x^{k+1}-y^{k+1} \rangle +
\frac{\|A\|}{2} \|x^{k+1}-y^{k+1}\|^2.
\eeq
In addition, notice that the objective of \eqref{step-45} is strongly convex. Then one can see from \eqref{step-45} that for all $x \in \cH(y^{k+1})$, 
\beqa
\frac12\|x-(y^{k+1}-t(A y^{k+1}-b))\|^2+t \tau\|x\|_1 & \ge &  \frac12\|x^{k+1}-(y^{k+1}-t(A y^{k+1}-b))\|^2  \nn \\
&& +t \tau\|x^{k+1}\|_1 + \frac12 \|x-x^{k+1}\|^2.  \label{err-bdd}
\eeqa
% where
%\[
%\bcH = \{x\in\Re^n: x_i = 0, \ i \in I_0(y^{k+1})\}.
%\]
Substituting $x=y^{k+1}$ into \eqref{err-bdd}, using $\nabla f(y^{k+1})=Ay^{k+1}-b$, and upon some manipulation, one has
\[
\langle \nabla f(y^{k+1}), x^{k+1}-y^{k+1}
\rangle \le \tau \|y^{k+1}\|_1 - \tau \|x^{k+1}\|_1 - t^{-1} \|x^{k+1}-y^{k+1}\|^2.
\]
Combining this inequality with \eqref{lip-ineq} and using the definition of $F$, we obtain that 
\beq \label{F-descent1}
F(y^{k+1}) \ge F(x^{k+1}) + \frac12\left(\frac2t-\|A\|\right) \|x^{k+1}-y^{k+1}\|^2.
\eeq
This together with $t\in(0,2/\|A\|)$ implies $F(x^{k+1}) \le F(y^{k+1})$. In addition, one can observe from  \eqref{step-44} that $F(y^{k+1}) \le F(x^k)$. It thus follows that $F(x^{k+1}) \le F(x^k)$.

Secondly, we claim that \eqref{step-44} cannot be executed at any $M+1$  consecutive iterations, where $M$ is defined in \eqref{M}.  
Suppose for contradiction that \eqref{step-44} is executed at iterations $i, i+1, \ldots, i+M$ iterations for some $i\ge 0$. This along with the updating scheme of $\CG$ method 4 implies that \eqref{measure} with $\eta=\heta$ does not hold at these $M+1$ iterations and thus $\heps$ is updated at them. By the updating scheme on $\heps$, it is not hard to verify 
that $\heps \le \min(\sqrt{\eta_0},1)\eps/\sqrt{\eta^*}$ at the iterations $i+M_1, \ldots, i+M$, where $M_1 = M- (n+1)$. We now show that the subroutine $\TPCG1$ executed in \eqref{step-44}  must terminate at a boundary point $x^{k+1}$ 
of the feasible region of problem \eqref{sub-qp} with $c=c^k$, $J_0=J^k_0$, $J_-=J^k_-$ and $J_+=J^k_+$ for all $i+M_1 \le k \le i+M-1$. Suppose not. By the termination 
criteria of $\TPCG1$, there exists some $i+M_1 \le k \le i+M-1$ such that 
\beq \label{y-opt}
\|\cP_\cH(Ay^{k+1}-b+c^k)\|_\infty \le \frac{\sqrt{(2t^{-1}-\|A\|)\|A\|}\ \heps}{(t^{-1}+\|A\|)\sqrt{n\heta}},
\eeq
where
\[
\cH = \{x\in\Re^n: x_j = 0, \ j \in J^k_0\}.
\]
 Observe from \eqref{step-44} that $J^k_0 = I_0(x^k) \subseteq I_0(y^{k+1})$ and moreover the nonzero components of $y^{k+1}$ share the same sign as the corresponding ones of $x^k$, which implies 
\[
%\ba{l}
%J^k_0 = I_0(x^k) \subseteq I_0(x^{k+1}), \\ [5pt]
c_i(y^{k+1};\tau)=c_i(x^k;\tau) = c^k_i, \quad\forall i \notin I_0(y^{k+1}).
%\ea
\]
Using these relations, \eqref{y-opt} and  \eqref{v}, we have 
\beq \label{sub-opt}
\left\|[v(y^{k+1})]_{I^\c_0(y^{k+1})}\right\|_\infty  =  \|\cP_\hcH(Ay^{k+1}-b+c(y^{k+1};\tau))\|_\infty \le \frac{\sqrt{(2t^{-1}-\|A\|)\|A\|}\ \heps}{(t^{-1}+\|A\|)\sqrt{n\heta}},
\eeq 
where  
\beq \label{hcH}
\hcH = \{x\in\Re^n: x_i = 0, \ i \in I_0(y^{k+1})\}.
\eeq
By the monotonicity of $\{F(x^l)\}$, we know that $F(x^k) \le F(x^0)$, which along with 
  $F(y^{k+1}) \le F(x^k)$ implies $F(y^{k+1}) \le F(x^0)$. Using this relation, \eqref{cg2-errbd} with $\eta=\eta^*$ and \eqref{sub-opt}, we obtain that 
\beqa
F(y^{k+1}) - F^*_{y^{k+1}}  &\le &  \frac{\eta^*}{2\|A\|}  \|[v(x)]_{I^\c_0(y^{k+1})}\|^2 \ \le \ \frac{n\eta^*}{2\|A\|} \|[v(x)]_{I^\c_0(y^{k+1})}\|^2_\infty \nn \\ [8pt]
&\le &\frac{\eta^*(2t^{-1}-\|A\|)\heps^2}{2\heta(t^{-1}+\|A\|)^2}, \label{Fys}
\eeqa
where $\eta^*$ is the smallest $\eta$ satisfying \eqref{cg2-errbd}. Notice that $I_0(y^{k+1}) \subseteq I_0(x^{k+1})$, which along with \eqref{Fxs} implies that 
  $F(x^{k+1}) \ge F^*_{y^{k+1}}$. In view of this, \eqref{F-descent1} and \eqref{Fys}, one has 
\[
\frac12\left(\frac2t-\|A\|\right) \|x^{k+1}-y^{k+1}\|^2  \ \le\  F(y^{k+1}) - F(x^{k+1}) \ \le \ F(y^{k+1}) - F^*_{y^{k+1}} 
\ \le\  \frac{\eta^*(2t^{-1}-\|A\|)\heps^2}{2\heta(t^{-1}+\|A\|)^2}.
\]
%\[
%\ba{lcl}
%\frac12\left(\frac2t-\|A\|\right) \|x^{k+1}-y^{k+1}\|^2 & \le & F(y^{k+1}) - F(x^{k+1}) \ \le \ F(y^{k+1}) - F^*_{y^{k+1}}, \\ [8pt]
%& \le & \frac{\eta^*(2t^{-1}-\|A\|)\heps^2}{2\heta(t^{-1}+\|A\|)^2}.
%\ea
%\]
It then follows that 
\[
\|x^{k+1}-y^{k+1}\| \le \frac{\heps\sqrt{\eta^*}}{\sqrt{\heta}(t^{-1}+\|A\|)}.
\] 
By the first-order optimality condition of \eqref{step-45}, one has
\[
0 \in \cP_\hcH(t^{-1}(x^{k+1}-y^{k+1})+\nabla f(y^{k+1})+\tau \partial \|x^{k+1}\|_1),
\]
where $\hcH$ is defined in \eqref{hcH}. It then follows from the last  
two relations that
\[
\ba{lcl}
\dist(0, \cP_\hcH(\nabla f(x^{k+1})+\tau \partial \|x^{k+1}\|_1)) & \le & \|t^{-1}(x^{k+1}-y^{k+1})+\nabla f(y^{k+1})-\nabla f(x^{k+1})\| \\ [8pt]
& \le & (t^{-1} + \|A\|) \|x^{k+1}-y^{k+1}\| \ \le \  \heps\sqrt{\eta^*/\heta}.
\ea
\]
This along with \eqref{v}, \eqref{hcH} and $I_0(y^{k+1}) \subseteq I_0(x^{k+1})$ implies that 
\beq \label{vI0c}
\left\|[v(x^{k+1})]_{I^\c_0(x^{k+1})}\right\| \le \heps\sqrt{\eta^*/\heta}.
\eeq 
By the above supposition, we know that \eqref{step-44} is executed at iteration $k+1$ 
and hence \eqref{measure} with $\eta=\heta$ does not hold at $x^{k+1}$. In view of this fact and \eqref{vI0c}, we further have 
\beq \label{vI0}
\left\|[v(x^{k+1})]_{I_0(x^{k+1})}\right\| \ \le \ \sqrt{\heta} \left\|[v(x^{k+1})]_{I^\c_0(x^{k+1})}\right\| \ \le \ \heps\sqrt{\eta^*}.
\eeq 
As shown above, $\heps \le \eps \min(\sqrt{\eta_0},1)/\sqrt{\eta^*}$ at iteration $k+1$. In view of this, $\heta \ge \eta_0$, \eqref{vI0c} and \eqref{vI0}, we have 
$\|v(x^{k+1})\|_\infty \le \eps$, which implies that $\CG$ method 4
terminates at $x^{k+1}$ and thus \eqref{step-44} will no longer be executed 
at iteration $k+1$. This contradicts the second supposition above. Therefore, 
TPCG1 must terminate at a boundary point $x^{k+1}$ 
of the feasible region of problem \eqref{sub-qp} with $c=c^k$, $J_0=J^k_0$, $J_-=J^k_-$ and $J_+=J^k_+$ for all $i+M_1 \le k \le i+M-1$. This together with the definitions of $J^k_0$, $J^k_-$ and $J^k_+$ implies that $I_0(x^k) \subsetneq I_0(x^{k+1})$ and 
hence $|I_0(x^k)| < |I_0(x^{k+1})|$ for all $i+M_1 \le k \le i+M-1$, which leads to $|I_0(x^{i+M})| \ge M-M_1=n+1$ and contradicts the trivial fact $|I_0(x^{i+M})| \le n$. 
Thus the above claim holds. 

From the second claim above, we can see that for each $\heta$, \eqref{step-44} is executed at most $M$ times between every two adjacent executions of \eqref{step-43}. By a similar argument as in the proof of Theorem \ref{CG3-convergence}, we know that  for each $\heta$, the number of executions of \eqref{step-43} is at most $2^n-1$. Thus, for each $\heta$ the number of executions of \eqref{step-44} is at most $M(2^n-1)$. It then follows that for each $\heta$, the total number of executions of \eqref{step-41}, \eqref{step-43} and \eqref{step-44} is at most $1+(M+1)(2^n-1)$.
In addition,  by a similar argument as in the proof of Theorem \ref{CG2v-convergence},  we know 
that  the number of distinct $\heta$ arising in this method is at most  $\max(\lceil (\log \eta^* - \log \eta_0)/\log \rho\rceil,0)+1$. The conclusion of the first statement immediately follows from these facts.

(ii) The proof the second statement is similar to that of Theorem \ref{CG1-convergence}.
\end{proof}

\gap

{\bf Remark 9:} From the proof of Theorem \ref{CG4-convergence}, we know that the subroutine $\TPCG1$ is called in $\CG4$ at most $M\max(\lceil (\log \eta^* - \log \eta_0)/\log \rho \rceil+1,1)(2^n-1)$ times, where $M$ is defined \eqref{M}. In view of this and Remark 2, one can observe that when $\eps=0$, the number of PCG iterations executed within $\CG4$ is at most $M\max(\lceil (\log \eta^* - \log \eta_0)/\log \rho \rceil+1,1)(2^n-1)(n+1)$. On the other hand, when $\eps>0$,  its number of PCG iterations depends on $\eps$ in $O(\log(1/\eps))$.

\section{The $l_1$ regularized least squares problem}
\label{l1-ls-prob}

In this section we consider a special class of problem \eqref{l1-qp} in the form of
\beq \label{l1-ls}
\bF^* = \min\limits_{x\in\Re^n} \bF(x) := \frac12 \|\bA x - \bb\|^2 + \tau \|x\|_1,
\eeq
which has important applications in compressed sensing and sparse regression. The $\CG$ methods proposed in Section \ref{algorithm} can be suitably applied to solve problem \eqref{l1-ls}. As shown in Section \ref{algorithm}, these methods are able to find an exact optimal solution of \eqref{l1-ls} within a finite number of iterations when the associated accuracy parameter $\eps$ is set to $0$, assuming no numerical errors. Despite this nice property, in practice it may be more interesting to find an approximate solution for two mains reasons. One is that only a subset of  real numbers can be represented precisely in computer and thus the truncation errors generally cannot be avoided.  Another reason is that even if an
exact optimal solution can be found, it may take many iterations to do that, which can be too expensive for large-scale problems. An
approximate solution usually suffices for practical purpose.

Given a tolerance parameter $\delta>0$, we are interested in applying the $\CG$ methods to find a $\delta$-optimal solution for problem \eqref{l1-ls}, that is,  a point $x_\delta$ such that $\bF(x_\delta)-\bF^* \le \delta$. Notice that $\bF^*$ is typically unknown. To terminate these methods properly, we need a suitable lower bound on $\bF^*$.
% criterion. One can observe that $\bF$ is descent along the solution sequence $\{x^k\}$ generated by the $\CG$ 
%methods when applied to \eqref{l1-ls}. Thus $\bF(x^k)$ can be used as an upper bound for the unknown
%optimal value $\bF^*$. 
In what follows, we derive some lower bounds for $\bF^*$. Before proceeding, we first establish a technical lemma.

\begin{lemma} \label{ls-errbdd}
Let $x^*$ be an arbitrary optimal solution of problem \eqref{l1-ls}. Then for any $x\in\Re^n$, the following inequalities hold:
\beqa
&&\left|\|\bA x-\bb\|^2 - \|\bA x^*-\bb\|^2\right| \ \le \ 2\left(\sqrt{\bF(x)}+\sqrt{\bF^*}\right)\sqrt{\bF(x)-\bF^*}, \nn \\ [8pt]
&& \left|\|x\|_1 - \|x^*\|_1\right| \ \le \ \frac{1}{\tau} \left[\bF(x)-\bF^*+\left(\sqrt{\bF(x)}+\sqrt{\bF^*}\right)\sqrt{\bF(x)-\bF^*}\right]. \label{1norm-gap}
\eeqa
\end{lemma}

\begin{proof}
One can observe that
\beq
\bF^* = \min\limits_{x\in\Re^n} \frac12 \|\bA x - \bb\|^2 + \tau \|x\|_1 ]
= \min\limits_{u \in \range(\bA)} \underbrace{\frac12 \|u -\bb\|^2 +
\tau \min\limits_{\bA x = u} \|x\|_1}_{\phi(u)}. \label{u-opt}
\eeq
Notice that $\phi$ is strongly convex in $u$. Hence, the latter problem in \eqref{u-opt} has a unique optimal solution, denoted by $u^*$. By the definition of $\phi$, one has
\[
 \bF^* \ \le \ \phi(\bA x^*) = \frac12 \|\bA x^* -\bb\|^2 + \tau \min\limits_{\bA x = \bA x^*} \|x\|_1
 \ \le \  \frac12 \|\bA x^* -\bb\|^2 + \tau \|x^*\|_1 = \bF^*,
\]
which implies $\phi(\bA x^*)=\bF^*$ and hence $u^*=\bA x^*$. In addition, by the strong convexity of $\phi$ and
the first-order optimality condition of \eqref{u-opt}, we obtain that
\[
\phi(u) - \bF^*  \ge  \frac12 \|u-u^*\|^2, \quad\quad \forall u \in \range(\bA).
\]
In view of this relation, the definition of $\phi$, and $u^*=\bA x^*$, one has
\[
\bF(x) - \bF^* \ \ge \ \phi(\bA x) - \bF^* \ \ge \ \frac12 \|\bA(x-x^*)\|^2, \quad\quad \forall x \in \Re^n,
\]
which yields
\beq \label{diff-Ax}
\|\bA(x-x^*)\| \le \sqrt{2(\bF(x) - \bF^* )}.
\eeq
Notice from the definition of $\bF$ that
\[
 \|\bA x^*-\bb\| \ \le \ \sqrt{2\bF^*}, \quad \|\bA x-\bb\| \ \le \ \sqrt{2\bF(x)}, \ \ \ \forall x\in\Re^n.
\]
Using these relations, \eqref{diff-Ax} and the definition of $\bF$,  we have
\[
\ba{rcl}
\left|\|\bA x-\bb\|^2 - \|\bA x^*-\bb\|^2\right| &\le& (\|\bA x -\bb\| + \|\bA x^*-\bb\|) \|\bA(x-x^*)\| \\ [8pt]
& \le & 2\left(\sqrt{\bF(x)}+\sqrt{\bF^*}\right)\sqrt{\bF(x)-\bF^*}, \\ [8pt]
\left|\|x\|_1 - \|x^*\|_1\right| &=& \frac{1}{\tau} \left|(\bF(x) - \frac12\|\bA x-\bb\|^2) - (\bF^* - \frac12\|\bA x^*-\bb\|^2) \right|, \\ [10pt]
&\le & \frac{1}{\tau}  \left[\bF(x) - \bF^* + \frac12\left|\|\bA x-\bb\|^2 - \|\bA x^*-\bb\|^2 \right|\right], \\ [10pt]
&\le & \frac{1}{\tau}  \left[\bF(x) - \bF^* + \left(\sqrt{\bF(x)}+\sqrt{\bF^*}\right)\sqrt{\bF(x)-\bF^*} \right].
\ea
\]
\end{proof}

\gap

In the following propositions we derive two computable lower bounds for $\bF^*$. 

\begin{proposition} \label{lower-bdd1}
Let $\{x^k\}$ be a sequence of approximate solutions to problem \eqref{l1-ls}, and let
\beq \label{F-low1}
\bF_{\low_1}(x^k) := \bF(x^k) - \langle \bA^T(\bA x^k-\bb), x^k \rangle - \tau \|x^k\|_1 + \min\left(1-\frac{\|\bA^T(\bA x^k-\bb)\|_\infty}{\tau},0\right)\bF(x^k),
\eeq
where $\langle \cdot, \cdot \rangle$ denotes the inner product of two associated vectors. Then the following statements hold:
\bi
\item[(i)] $\bF^* \ge \bF_{\low_1}(x^k)$ for all $k$;
\item[(ii)] If $\bF(x^k) \to \bF^*$, then $\bF_{\low_1}(x^k) \to \bF^*$.
\ei
\end{proposition}

\begin{proof} Let $x^*$ be an arbitrary optimal solution of \eqref{l1-ls}.

(i) Using the definition of $\bF$, we have
\[
\tau \|x^*\|_1 \le \bF(x^*) \le \bF(x^k),
\]
which implies
\beq \label{bdd-soln}
\|x^*\|_1 \le \tau^{-1} \bF(x^k), \quad\quad \forall k.
\eeq
It then follows that  $x^* \in \Delta_k :=\{x\in\Re^n: \|x\|_1 \le \tau^{-1} \bF(x^k)\}$.
By the convexity of $\|\bA \cdot - \bb\|^2/2$, one
has
\[
\bF(x) = \frac12 \|\bA x-\bb\|^2+\tau \|x\|_1 \ \ge \ \frac12 \|\bA x^k-\bb\|^2 + \langle \bA^T(\bA x^k-\bb), x- x^k \rangle + \tau \|x\|_1,
\]
Using this and the fact $x^*\in\Delta_k$, we have 
\[
\ba{lcl}
\bF^* &=&  \min\limits_{x\in \Delta_k} \bF(x) \ \ge \ \min\limits_{x\in \Delta_k} \left\{\frac12 \|\bA x^k-\bb\|^2 + \langle \bA^T(\bA x^k-\bb), x- x^k \rangle + \tau \|x\|_1\right\}, \\ [15pt]
&=& \frac12 \|\bA x^k-\bb\|^2 - \langle \bA^T(\bA x^k-\bb), x^k \rangle + \min\limits_{x\in \Delta_k} \left\{ \langle \bA^T(\bA x^k-\bb), x \rangle + \tau \|x\|_1\right\}, \\ [15pt]
&=& \bF(x^k) - \langle \bA^T(\bA x^k-\bb), x^k \rangle - \tau \|x^k\|_1 +\min\limits_{x\in \Delta_k} \left\{ \langle \bA^T(\bA x^k-\bb), x \rangle + \tau \|x\|_1\right\}, \\ [15pt]
&=& \bF(x^k) - \langle \bA^T(\bA x^k-\bb), x^k \rangle - \tau \|x^k\|_1 + \min\left(1-\frac{\|\bA^T(\bA x^k-\bb)\|_\infty}{\tau},0\right)\bF(x^k),
\ea
\]
which together with \eqref{F-low1} implies that statement (i) holds.

(ii) Suppose $\bF(x^k) \to \bF^*$. Recall that $x^*$ is an optimal solution of \eqref{l1-ls}. By the first
optimality condition of \eqref{l1-ls} at $x^*$, one has 
\beq \label{1st-opt1}
0 \in \bA^T(\bA x^*-\bb)+ \tau\partial\|x^*\|_1
\eeq
and hence $\|\bA^T(\bA x^*-\bb)\|_\infty \le \tau$.
In view of \eqref{diff-Ax} and the assumption $\bF(x^k) \to \bF^*$, we have $\|\bA(x^k-x^*)\| \to 0$ and hence
$\bA x^k \to \bA x^*$, which implies
\beqa
&&\langle \bA^T(\bA x^k-\bb), x^k \rangle = \langle \bA x^k-\bb, \bA x^k \rangle \ \to \
\langle \bA x^*-\bb, \bA x^* \rangle = \langle \bA^T(\bA x^*-\bb),  x^* \rangle, \label{lim1} \\ [5pt]
&& \min\left(1-\frac{\|\bA^T(\bA x^k-\bb)\|_\infty}{\tau},0\right) \ \to \
\min\left(1-\frac{\|\bA^T(\bA x^*-\bb)\|_\infty}{\tau},0\right) \ = \ 0, \label{lim2}
\eeqa
where the last equality is due to $\|\bA^T(\bA x^*-\bb)\|_\infty \le \tau$. In addition, by the the assumption $\bF(x^k) \to \bF^*$ and \eqref{1norm-gap}, we have $\|x^k\|_1 \to \|x^*\|_1$. Also, by \eqref{1st-opt1}, one can observe that $x^*$ is an optimal solution to the problem
\[
\min\limits_x \langle \bA^T(\bA x^*-\bb), x\rangle + \tau \|x\|_1.
\]
Notice that the objective of this problem is positive homogeneous. Hence,  its optimal value is $0$. It then follows that 
\[
\langle \bA^T(\bA x^*-\bb), x^*\rangle + \tau \|x^*\|_1 = 0.
\]
Using this relation, $\bF(x^k) \to \bF^*$, $\|x^k\|_1 \to \|x^*\|_1$, \eqref{lim1}, \eqref{lim2} and taking limits on both sides
of \eqref{F-low1}, we have $\bF_{\low_1}(x^k) \to \bF^*$.
\end{proof}

\gap

As seen from Proposition \ref{lower-bdd1}, $\bF_{\low_1}(x^k)$ is a suitable lower bound for $\bF^*$. We next derive another lower bound.

\begin{proposition} \label{lower-bdd2}
(i) Let $\{x^k\}$ be a sequence of approximate solutions to problem \eqref{l1-ls}, $v$ be defined in \eqref{v} with $f(x)=\|\bA x-\bb\|^2/2$, and let 
\beq \label{F-low2}
\bF_{\low_2}(x^k) := \bF(x^k) (1-\tau^{-1}\|v(x^k)\|_\infty)-\langle v(x^k), x^k \rangle.
\eeq
The following statements hold:
%\beq \label{F-low2}
%\bF^* \ \ge \ \bF_{\low_2}(x^k) := \bF(x^k) (1-\tau^{-1}\|v(x^k)\|_\infty)-\langle v(x^k), x^k \rangle.
%\eeq
\bi
\item[(i)] $\bF^* \ge \bF_{\low_2}(x^k)$ for all $k$.
\item[(ii)] If $\{x^k\}$ is bounded and $v(x^k) \to 0$, then $\bF_{\low_2}(x^k) \to \bF^*$.
\ei
\end{proposition}

\begin{proof}
(i) Let $x^*$ be an arbitrary optimal solution of \eqref{l1-ls}. In view of \eqref{v}, one knows that $v(x^k) \in \partial \bF(x^k)$. Using this relation and \eqref{bdd-soln}, we obtain that
\[
\ba{lcl}
\bF^* & = & \bF(x^*) \ge \bF(x^k) + \langle v(x^k), x^*-x^k \rangle  \ = \ \bF(x^k) - \langle v(x^k), x^k \rangle + \langle v(x^k), x^*\rangle \\ [8pt]
&\ge & \bF(x^k) -\langle v(x^k), x^k \rangle- \|v(x^k)\|_\infty  \|x^*\|\ \ge \ \bF(x^k) (1-\tau^{-1}\|v(x^k)\|_\infty)-\langle v(x^k), x^k \rangle,
\ea
\]
and hence statement (i) holds.

(ii) Assume that $\{x^k\}$ is bounded and $v(x^k) \to 0$. Then $\{\bF(x^k)\}$ is bounded. It follows from this and Lemma \ref{errbdd-F} with $F=\bF$ and $F^*=\bF^*$ that there exists some $\eta>0$ such that $\bF(x^k)-\bF^* \le \eta \|v(x^k)\|$ for all $k$. This along with $v(x^k) \to 0$ and the definition of $\bF^*$ implies $\bF(x^k) \to \bF^*$. The conclusion of this statement immediately follows from this relation, \eqref{F-low2}, $v(x^k) \to 0$ and the boundedness of $\{x^k\}$.
\end{proof}

\gap

We next propose a practical termination criterion for the $\CG$ methods 
by using the above two lower bounds on $\bF^*$. This criterion may also be 
useful for any other methods for solving problem \eqref{l1-ls}.  
%that are given in \eqref{F-low1} and \eqref{F-low2}.

\begin{theorem} \label{term-criterion}
Let $x^0$ be an arbitrary initial point for the $\CG$ methods. Given any $\delta \ge 0$, let
\beq \label{eps}
\eps = \tau \delta/(2\bF(x^0)).
\eeq
Let $\{x^k\}$ be the sequence generated by the $\CG$ methods applied to problem 
\eqref{l1-ls} starting at $x^0$ with the above $\eps$ as the input accuracy parameter. Suppose that these methods are terminated once 
 $\|v(x^k)\|_\infty \le \eps$ or
\beq \label{F-gap1}
\bF(x^k) - \max(\bF_{\low_1}(x^k),\bF_{\low_2}(x^k)) \le \delta
\eeq
holds. Then the $\CG$ methods terminate within a finite number of iterations at a $\delta$-optimal solution $x^k$ of \eqref{l1-ls}, that is, $\bF(x^k)-\bF^* \le \delta$.
\end{theorem}

\begin{proof}
In view of Theorems \ref{CG1-convergence}-\ref{CG4-convergence}, one can see that at least one of the above termination criteria must be satisfied at some iteration $k$. Also, one can observe that $\bF(x^k) \le \bF(x^0)$. We next show that $\bF(x^k)-\bF^* \le \delta$ holds by considering two cases.

Case 1): $\|v(x^k)\|_\infty \le \eps$ holds. By the definition of $\bF$, one has 
$\tau \|x^k\|_1 \le \bF(x^k)$, which yields $\|x^k\|_1 \le \tau^{-1} \bF(x^k)$. 
 Using this relation and \eqref{F-low2}, we obtain that 
 \[
\bF_{\low_2}(x^k) \ge   \bF(x^k) (1-\tau^{-1}\|v(x^k)\|_\infty)-\|v(x^k)\|_\infty \|x^k\|_1 \ge \bF(x^k) (1-2\tau^{-1}\|v(x^k)\|_\infty).
 \] 
 It follows from this, $\bF^* \ge \bF_{\low_2}(x^k)$, $\|v(x^k)\|_\infty \le \eps$,  $\bF(x^k) \le \bF(x^0)$ and \eqref{eps} that 
\[
\bF(x^k) - \bF^* \ \le \  2\tau^{-1} \bF(x^k) \|v(x^k)\|_\infty \ \le \  2\tau^{-1} \bF(x^0) \|v(x^k)\|_\infty \ \le \  2\tau^{-1} \bF(x^0) \eps \ = \  \delta.
\]

Case 2): \eqref{F-gap1} holds. It follows from Propositions \ref{lower-bdd1} and \ref{lower-bdd2}
that
\[
\max(\bF_{\low_1}(x^k),\bF_{\low_2}(x^k)) \ge \bF^*,
\]
which together with \eqref{F-gap1} implies $\bF(x^k) - \bF^* \le \delta$.
\end{proof}

\section{Numerical results} 
\label{result}

In this section we conduct numerical experiments to test the performance of the $\CG$ methods proposed in Section \ref{algorithm} and compare them with some closely related methods, which are  the fast iterative shrinkage-thresholding algorithm (FISTA) \cite{Beck} with a constant stepsize $1/L$, the interleaved ISTA-CG method (iiCG) \cite{Byrd}, and the nonmonotone proximal gradient method (NPG) \cite{WrNoFi09} without continuation. It shall be mentioned that four iiCG algorithms were proposed in \cite{Byrd} that differ in the choice of stepsize (fixed or variable) and proximal step (ISTA or reduced ISTA). As observed in our experiment, the iiCG algorithm with the ISTA step and the variable stepsize determined by 
the subroutine ISTA-BB-LS described in \cite{Byrd} outperforms the others. Therefore, we only include this one in our numerical comparison. In addition, extensive numerical experiments conducted in \cite{Byrd} demonstrate that the iiCG methods are competitive with the other state-of-the-art codes including PSSgb \cite{Sch10}, N83 \cite{BeCaGr11}, pdNCG \cite{FoGo14}, SALSA \cite{AfBiFi10}, TWIST \cite{BiFi07}, FPC\_AS \cite{HaYiZh07}, l1\_ls \cite{KiKoLuBo07} and YALL1 \cite{DeYiZh11}. Therefore, 
we will not compare our methods with them in this paper.

 We choose $x^0=0$ as the initial point for all methods and terminate them once the termination criterion \eqref{F-gap1} with some $\delta$ is met. It then follows from Theorem \ref{term-criterion} that the approximate solution obtained by them is a $\delta$-optimal solution of \eqref{l1-ls}. For the $\CG$ methods, we set $t=2/(\|\bA\|^2+10^{-4})$, $\eta_0=\kappa^2(\bA)$, $\rho=10$, $\xi=0.5$ and $\eps=\tau \delta/(2\bF(x^0))$, where $\tau$ and $\bF$ are given in \eqref{l1-ls}. For FISTA, iiCG and NPG, we choose the same parameters as mentioned in \cite{Beck,Byrd,WrNoFi09} except $M$. In particular, we set $L=\|\bA\|^2$ (i.e., the Lipschitz constant of the gradient of $\|\bA x -\bb\|^2/2$) for FISTA, $c=10^{-4}$ and $\xi=0.005$ for iiCG, and $\sigma=10^{-2}$, $\alpha_{\max}=1/\alpha_{\min}=10^{30}$ and $\eta=2$ for NPG. For the purpose of comparison of these methods, we do not include the time for evaluating these parameters in the CPU time reported below. The codes of all methods are written in Matlab. All computations are performed by Matlab R2015b running on a Dell Optiplex 9020 personal computer with a 3.40-GHz Intel Core i7-4770 processor and 32 GB of RAM. 
  
In the first experiment we test the performance of the $\CG$ methods, 
particularly, $\CG2_v$, $\CG3$ and $\CG4$ and also compare them with FISTA, iiCG  and NPG for solving problem \eqref{l1-ls} in which $\bA$ is well-conditioned, namely, with a small $\kappa(\bA)$.  We randomly generate $\bA \in \Re^{m \times n}$ and  $\bb\in \Re^m$ in the same manner as described in $l_1$-magic \cite{CanRom05}. In particular, given $\sigma >0$ and positive integers $m$, $n$, $s$ with $m < n$ and $s< n$, we first generate a matrix $W\in\Re^{n\times m}$ with entries randomly chosen 
from a standard normal distribution. We then compute an orthonormal basis, denoted by $B$, for the range space 
of $W$, and set $\bA=B^T$. In addition, we randomly generate a vector $\tx \in \Re^n$ with 
only $s$ nonzero components that are $\pm 1$, and generate a vector $v\in\Re^{m}$ 
with entries randomly chosen from a standard normal distribution. Finally, we set $\bb = \bA\tx+\sigma v$. In 
particular, we choose $\sigma=10^{-5}$ for all instances. In addition, we 
set $\tau=0.1$ for problem \eqref{l1-ls}, and $M=5$ for iiCG and NPG that is the same as in \cite{Byrd,WrNoFi09}.

%We choose $x^0=0$ as the initial point for all methods, and set $\eta_0=\kappa(A)$, $\rho=10$ and $\xi=0.5$ for them. 
In Table \ref{cg-res1} we present the computational results obtained by those methods based on the termination criterion \eqref{F-gap1} with $\delta=10^{-2}$. The parameters $m$, $n$ and $s$ of each instance are listed in the first three columns, respectively. We observe from the experiment that the 
approximate solutions found by these methods have almost same cardinality and nearly equal objective value. We only report in the rest of columns of Table \ref{cg-res1} the average cardinality of their approximate solutions and also the CPU times (in seconds) of each method.     
 From Table \ref{cg-res1}, one can see that $\CG4$ consistently outperforms  $\CG2_\v$ and $\CG3$ in terms of CPU time. Though $\CG4$ is slower than FISTA, it is competitive with iiCG and NPG.   
%In addition, as the accuracy parameter $\delta$ becomes smaller, the CPU time of these method 
%does not increase much. Therefore, these methods are suitable for finding highly accurate solutions.  

%\begin{table}[t!]
%\caption{GCG methods for \eqref{l1-ls} with well-conditioned $\bA$}
%\centering
%\label{cg-res1}
%%\begin{center}
%%\begin{small}
%\begin{tabular}{|rrr||c||ccc|}
%\hline
%\multicolumn{3}{|c||}{Problem} &  \multicolumn{1}{c||}{Cardinality} & 
% \multicolumn{3}{c|}{CPU Time} \\
%\multicolumn{1}{|c}{m} & \multicolumn{1}{c}{n} &  \multicolumn{1}{c||}{s}
%& \multicolumn{1}{c||}{} 
%& \multicolumn{1}{c}{\sc $\CG2_\v$} & \multicolumn{1}{c}{\sc $\CG3$} 
%& \multicolumn{1}{c|}{\sc $\CG4$} \\
%\hline
%120 & 512 & 20 &  24   & 0.01 & 0.01 & 0.01 \\
%240 & 1024 & 40 & 56  & 0.05 & 0.05 & 0.03 \\
%360 & 1536 & 60 & 77   & 0.12 & 0.17 & 0.07 \\
%480 & 2048 & 80 & 113   & 0.31 & 0.52 & 0.15 \\
%600 & 2560 & 100 & 137 & 0.53 & 0.25 & 0.22\\
%720 & 3072 & 120 & 155  & 0.76 & 2.35 & 0.43 \\
%840 & 3584 & 140& 188   & 1.66 & 3.48 & 0.59 \\
%960 & 4096 & 160& 212   & 1.70 & 5.23 & 0.73 \\
%1080 & 4608 & 180& 247   & 2.83 & 7.19 & 1.04 \\
%1200 & 5120 & 200& 269   & 3.09 & 9.88 & 1.28 \\
%\hline
%\end{tabular}
%%\end{small}
%%\end{center}
%\end{table}

\begin{table}[t!]
\caption{Comparison on the methods for \eqref{l1-ls} with well-conditioned $\bA$}
\centering
\label{cg-res1}
%\begin{center}
%\begin{small}
\begin{tabular}{|rrr||c||cccccc|}
\hline
\multicolumn{3}{|c||}{Problem} &  \multicolumn{1}{c||}{Cardinality} & 
 \multicolumn{6}{c|}{CPU Time} \\
\multicolumn{1}{|c}{m} & \multicolumn{1}{c}{n} &  \multicolumn{1}{c||}{s}
& \multicolumn{1}{c||}{} & \multicolumn{1}{c}{\sc $\CG2_\v$} 
& \multicolumn{1}{c}{\sc $\CG3$} & \multicolumn{1}{c}{\sc $\CG4$} 
& \multicolumn{1}{c}{\sc FISTA}
& \multicolumn{1}{c}{\sc iiCG} & \multicolumn{1}{c|}{\sc NPG}  \\
\hline
120 & 512 & 20 &  24   & 0.01 & 0.01 & 0.01 & 0.00 & 0.01 & 0.01  \\
240 & 1024 & 40 & 56  & 0.06 & 0.05 & 0.03 & 0.00 & 0.03 & 0.02 \\
360 & 1536 & 60 & 77   & 0.13 & 0.18 & 0.07 &0.01 & 0.07 &0.07  \\
480 & 2048 & 80 & 113   & 0.28 & 0.50 & 0.17 &0.02 & 0.14 &0.14 \\
600 & 2560 & 100 & 137 & 0.52 & 1.23 & 0.26&0.05 & 0.26 & 0.24 \\
720 & 3072 & 120 & 155  & 0.82 & 2.41 & 0.46 &0.08& 0.42 &0.34  \\
840 & 3584 & 140& 188   & 1.58 & 3.47 & 0.57 &0.14&0.54 &0.47  \\
960 & 4096 & 160& 212   & 1.67 & 5.19 & 0.73 &0.18 & 0.79 &0.65 \\
1080 & 4608 & 180& 247   & 2.93 & 7.17 & 1.10 & 0.25& 1.03 & 0.88  \\
1200 & 5120 & 200& 269   & 3.10 & 9.90 & 1.34 & 0.32 & 1.34 & 1.19  \\
\hline
\end{tabular}
%\end{small}
%\end{center}
\end{table}

We next compare the performance of the aforementioned methods for 
solving problem \eqref{l1-ls} in which $\bA$ is ill-conditioned, namely, with a large $\kappa(\bA)$. The data $\bA$ and $\bb$ are similarly generated as above except that $\bA$ is set to $B^TD$ rather than $B^T$, where $D$ is an $n \times n$ diagonal matrix whose $i$th diagonal entry is $\min(i^2,10^6)$ for all $i$. In addition, we set $\tau=1$ for \eqref{l1-ls}, and $M=15$ for iiCG and NPG that 
generally gives better performance than the other choices of $M$ in this experiment.  We present in Table \ref{cg-res2}  the computational results obtained by those 
methods based on the termination criterion with $\delta=10^{-2}$. As in the 
first experiment, the approximate solutions found by those methods have 
almost same cardinality and nearly equal objective value. We thus 
only report the average cardinality of these approximate solutions and also the 
CPU times (in seconds) of each method. One can see from Table \ref{cg-res2} 
that $\CG2_\v$ consistently outperforms the other methods in terms of CPU 
time and appears to be more favorable for solving problem \eqref{l1-ls} with 
ill-conditioned $\bA$. 

\begin{table}[t!]
\caption{Comparison on the methods for \eqref{l1-ls} with ill-conditioned $\bA$}
\centering
\label{cg-res2}
%\begin{center}
%\begin{small}
\begin{tabular}{|rrr||c||cccccc|}
\hline
\multicolumn{3}{|c||}{Problem} &  \multicolumn{1}{c||}{Cardinality} & 
 \multicolumn{6}{c|}{CPU Time} \\
\multicolumn{1}{|c}{m} & \multicolumn{1}{c}{n} &  \multicolumn{1}{c||}{s}
& \multicolumn{1}{c||}{} & \multicolumn{1}{c}{\sc $\CG2_\v$} 
& \multicolumn{1}{c}{\sc $\CG3$} & \multicolumn{1}{c}{\sc $\CG4$} 
& \multicolumn{1}{c}{\sc FISTA}
& \multicolumn{1}{c}{\sc iiCG} & \multicolumn{1}{c|}{\sc NPG}  \\
\hline
120 & 512 & 20 &  36   & 0.17 & 0.39 & 0.23 & 0.36 & 0.60 & 6.23  \\
240 & 1024 & 40 & 57  & 0.85 & 2.02 & 1.22 & 1.06& 2.94 & 16.88  \\
360 & 1536 & 60 & 77  & 2.48 & 6.81 & 3.35 & 2.72& 3.37 & 7.02  \\
480 & 2048 & 80 & 114  & 6.19 & 22.27 & 9.66 & 8.26& 11.40 & 50.74  \\
600 & 2560 & 100 & 140 & 16.72 & 45.82 & 20.63 & 29.17& 32.00 & 251.68 \\
720 & 3072 & 120 & 154 & 27.67 & 89.63 & 39.82 & 55.43& 86.16 & 171.91  \\
840 & 3584 & 140& 186  & 49.55 & 148.23 & 67.83 & 98.76& 201.14 & 84.37  \\
960 & 4096 & 160& 212  & 96.08 & 233.86 & 110.54 & 153.72& 298.99 & 635.62  \\
1080 & 4608 & 180& 261  & 159.04 & 351.05 & 174.62 & 230.58 &406.29 & 463.51 \\
1200 & 5120 & 200& 285  & 227.87 & 468.09 & 272.06 & 320.31& 720.73 &6198.07 \\
\hline
\end{tabular}
%\end{small}
%\end{center}
\end{table}

\section{Concluding remarks}
\label{conclude}

In this paper we proposed generalized CG (GCG) methods for solving $\ell_1$ regularized convex QP.  When the tolerance parameter $\eps$ is set to $0$,  our GCG methods terminate at an optimal solution in a finite number of iterations, assuming no numerical errors.  We also show that our methods are capable of finding an approximate solution of the problem by allowing some inexactness on the execution of the CG subroutine. Numerical results demonstrate that our methods are very favorable for solving ill-conditioned problems. 

It can be observed that the main computation of our GCG methods lies in 
executing PCG iterations over a sequence of subspaces. Given any $\eps>0$,  one can see from Remarks 4-9 that the number of PCG iterations of these methods for finding an approximate solution $x$ satisfying $\|v(x)\| \le \sqrt{\eps}$ depends on $\eps$ in $O(\log(1/\eps))$. It follows from this and Theorem \ref{errbdd-F} that 
the number of PCG iterations for finding an $\eps$-optimal solution by 
 these methods depends on $\eps$ in $O(\log(1/\eps))$. 
Therefore, the overall operation cost of the GCG methods for finding an $\eps$-optimal solution depends on $\eps$ in $O(\log(1/\eps))$, which is superior to the accelerated proximal gradient (APG) method \cite{Beck,Nest13} whose operation cost  depends on $\eps$ in $O(1/\sqrt{\eps})$. In addition, at each iteration 
the main computation of APG is on full-dimensional matrix-vector products while that of PCG is on lower-dimensional matrix-vector products.    

%It can be observed that the subroutine $\TPCG1$ is the core subroutine executed in our GCG methods. One can also see from Theorem \ref{TPCG1-converge} that the tolerance parameter $\eps$ affects the number of iterations of $\TPCG1$  by a factor of $\log(1/\eps)$. Therefore, the overall arithmetic operation costs of our GCG methods 
%for finding an $\eps$-optimal solution depend on $\eps$ in a factor $\log(1/\eps)$, 
%which is superior to the accelerated proximal gradient method \cite{Beck,Nest13} that depends on $\eps$ in a factor  $1/\sqrt{\eps}$.   

Our GCG methods can be extended to solve the box-constrained convex QP  \eqref{bQP} by properly modifying the definitions of $v(\cdot)$, $\vp(\cdot)$, $I_0(\cdot)$, $I_+(\cdot)$, $I_-(\cdot)$, $I^0_0(\cdot)$,  $I^+_0(\cdot)$, $I^-_0(\cdot)$ and also the subroutines $\TPCG1$ and $\TPCG2$. The similar convergence results as in this paper can also be established. Due to the length limitation, such an extension is left in a separate paper. 

\appendix

\section{Appendix: The CG method for convex QP}

In this appendix we study CG method for solving (possibly not strongly) convex quadratic programming:
\beq \label{qp}
f^* = \inf\limits_{x\in\Re^n} f(x) := \frac12 x^T A x - b^T x,
\eeq
where $A \in \Re^{n\times n}$  is symmetric positive semidefinite, $b\in\Re^n$ and $f^* \in [-\infty,\infty)$. It is well-known that \eqref{qp} has at least an optimal solution and $f^*$ is finite if $b\in\range(A)$,  and $f^*=-\infty$ otherwise. The standard CG method (e.g., see \cite[Chapter 5]{NoWr06}) is proposed mainly for the case where $A$ is symmetric positive definite. In order to make the CG method applicable to the general convex QP \eqref{qp}, one has to modify the termination criterion of the standard CG method, which is usually in terms of the gradient of the objective of \eqref{qp}. The resulting CG method is presented as follows. 

 \gap

\noindent {\bf CG method for \eqref{qp}:}

\vspace{.1in}

\noindent{\bf Input:} $A$, $b$, $x^0$.

\vspace{.1in}

\noindent Set $r^0 = Ax^0-b$, $p^0=-r^0$, $k=0$.

\vspace{.1in}

\noindent {\bf while} $Ap^k \neq 0$
\beqa
&& \alpha_k = \frac{\|r^k\|^2}{(p^k)^T A p^k}; \label{alphak} \\ [5pt]
&& x^{k+1} = x^k + \alpha_k p^k; \label{xk} \\
&& r^{k+1} = r^k + \alpha_k A p^k; \label{rk} \\ [5pt]
&& \beta_{k+1} = \frac{\|r^{k+1}\|^2}{\|r^k\|^2}; \label{betak} \\ [5pt]
&& p^{k+1} = -r^{k+1} + \beta_{k+1} p^k; \label{pk} \\
&& k \leftarrow k+1. \nn
\eeqa
{\bf end (while)}

\gap

%It is well-known that \eqref{qp} has at least an optimal solution if $b\in\range(A)$,  and $f^*=-\infty$ otherwise.  
%We make the following assumption throughout this section.
%
%\begin{assumption} \label{assump-qp}
% $b \in \range(A)$.
%\end{assumption}
%
%This Assumption implies that problem \eqref{qp} has at least an optimal solution and $f^*$ is finite. 
The convergence of the CG method for \eqref{qp} has been well studied  when $A$ is symmetric positive definite (e.g., see \cite{NoWr06} and the references therein). Nevertheless, for the case where $A$ is symmetric positive semidefinite, it has only been partially studied in the literature (e.g., see \cite{KaNa72,Ka88,Ha01}). As a self-contained reference, we next provide a more comprehensive study on the convergence of the CG method for problem \eqref{qp} with $A$ being symmetric positive \textit{semidefinite}.

\begin{theorem} \label{cg-prop}
The following statements hold for the above CG method:
\bi
\item[(i)] The CG method is well defined.
\item[(ii)] Suppose that the CG method has not yet terminated at the $k$th iterate, that is, $Ap^k \neq 0$  for some $k \ge 0$. Then the following properties hold:
\beqa 
(r^i)^T r^j =0, \quad\quad \mbox{for}\  i \neq j, \ i,j=0, 1, \ldots, k,  \nn\\
\Span\{r^0, r^1, \ldots, r^k\} = \Span\{r^0,Ar^0,\ldots,A^k r^0\},   \nn  \\
\Span\{p^0, p^1, \ldots, p^k\} = \Span\{r^0,Ar^0,\ldots,A^k r^0\},   \label{p-span} \\
(p^i)^T A p^j=0, \quad\quad \mbox{for}\  i \neq j, \ i,j=0, 1, \ldots, k, \label{p-orth} 
\eeqa  
where $\Span$ denotes the subspace spanned by the associated vectors. 
\ei
\end{theorem}

\begin{proof}
(i) We first claim that 
\beq \label{ap-r}
\mbox{if} \ Ap^k \neq 0\ \mbox{for some} \ k \ge 0, \ \mbox{then} \ r^k \neq 0. 
\eeq
Indeed, since $p^0=-r^0$, it is clear that $Ap^0\neq 0$ yields $r^0 \neq0$. Suppose for contradiction that $Ap^k \neq 0$ for some $k>0$ but $r^k=0$. It then follows from \eqref{betak} 
that $\beta_k =\|r^k\|^2/\|r^{k-1}\|^2=0$. This together with \eqref{pk} and $r^k=0$ 
gives $p^k = -r^k+\beta_k p^{k-1}=0$, which contradicts the assumption $Ap^k \neq 0$ 
and thus \eqref{ap-r} holds. In addition, since $A$ is symmetric positive semidefinite, one can see 
that $(p^k)^TAp^k \neq 0$ if and only if $Ap^k \neq 0$. By this fact, \eqref{ap-r} and an inductive argument, it is not hard to see that the CG method is well defined.  

(ii) The proof of statement (ii) is identical to that of  \cite[Theorem 5.3]{NoWr06}.
\end{proof}

\gap

The following result shows that the CG method terminates in a finite number of iterations and produces either an optimal solution of \eqref{qp} or a direction along which $f$ is unbounded below.  

\begin{theorem} \label{cg-unbdd} 
Let $\ell =\rank(A)$. The following properties hold for the CG method:
\bi
\item[(i)] If $b\in\range(A)$, the CG method terminates at some iteration $0\le k \le \ell$ and $x^k$ is an optimal solution of \eqref{qp};
\item[(ii)] If $b\notin \range(A)$,  the CG method terminates at some iteration $0\le k \le \ell+1$ and $f(x^k+\alpha p^k) \to -\infty$ as $\alpha \to \infty$.  
\ei
\end{theorem}

\begin{proof}
(i) Assume that $b\in\range(A)$. Suppose for contradiction that the CG method does not terminate for all $0 \le k \le \ell$. It then follows that $Ap^k \neq 0$ and hence $p^k \neq 0$ 
for every $0 \le k \le \ell$. From \eqref{p-orth} we also know that $\{p^i\}^{\ell}_{i=0}$ 
are conjugate with respect to $A$. Thus $\{p^i\}^{\ell}_{i=0}$  are linearly independent. 
In view of this and \eqref{p-span}, one can see that the dimension of $\Span\{r^0,Ar^0,\ldots,A^{\ell} r^0\}$ is $\ell+1$. On  the other hand , by the assumption $b\in\range(A)$, 
we observe that $r^0 = Ax^0-b \in\range(A)$. It yields $\Span\{r^0,Ar^0,\ldots,A^{\ell} r^0\} \subseteq \range(A)$, which along with $\ell=\rank(A)$ implies that the dimension of $\Span\{r^0,Ar^0,\ldots,A^{\ell} r^0\}$ is at most $\ell$. This  leads to a contradiction. Thus the CG method terminates at some $0\le k\le \ell$ with $Ap^k=0$. We next show that $x^k$ is an optimal solution of \eqref{qp}. Since \eqref{qp} is a convex problem, it suffices to show $\nabla f(x^k)=0$. Notice that $r^k=Ax^k - b$, which along with the assumption $b\in\range(A)$ implies that $r^k=A\xi$ for some vector $\xi$. By this and $Ap^k=0$, we have 
$(r^k)^Tp^k=\xi^T Ap^k=0$. Since the exact line search is performed at each 
iteration of CG, one has $(r^k)^T p^{k-1}=0$. In addition, it follows from \eqref{pk} that 
$r^k = - p^k+\beta_k p^{k-1}$. Using these relations, we have 
$
\|r^k\|^2 = (r^k)^T (- p^k+\beta_k p^{k-1})=0.
$
Hence, $\nabla f(x^k)=r^k=0$ as desired.

 (ii) Assume that $b\notin \range(A)$. Suppose for contradiction that the CG method does not terminate for all $0 \le k \le \ell+1$. By a similar argument as in 
statement (i), one knows that $\{p^i\}^{\ell+1}_{i=0}$  are linearly independent. Using this and \eqref{p-span}, we see that the dimension of $\Span\{r^0,Ar^0,\ldots,A^{\ell+1} r^0\}$ is $\ell+2$. In addition, notice that  
$\Span\{r^0,Ar^0,\ldots,A^{\ell+1} r^0\} \subseteq \Span\{r^0 \cup \range(A)\}$. 
This and $\ell=\rank(A)$ imply that the dimension of $\Span\{r^0,Ar^0,\ldots,A^{\ell+1} r^0\}$ is at most $\ell+1$, which leads to a contradiction. Thus the CG method terminates at some $0\le k\le \ell+1$ with $Ap^k=0$. It remains to show that $f(x^k+\alpha p^k) \to -\infty$ as $\alpha \to \infty$. By the same argument as above, one has $(r^k)^T p^{k-1}=0$. Also, by the assumption $b\notin \range(A)$, we know that $r^k \neq 0$. In addition,  notice from \eqref{pk} that $p^k = -r^k + \beta_k p^{k-1}$. It follows from these relations that $(r^k)^T p^k = -\|r^k\|^2 <0$. In view of this and $Ap^k=0$, we have
\[
f(x^k+\alpha p^k) = f(x^k) + \alpha(r^k)^T p^k  + \frac12 \alpha^2 (p^k)^T A p^k = f(x^k) + \alpha(r^k)^T p^k 
\ \to -\infty \ \mbox{as} \ \alpha \to \infty.
\]
\end{proof}

\gap

We next study some further convergence properties of the CG method under the following assumption. 
\begin{assumption} \label{assump-qp}
 $b \in \range(A)$.
\end{assumption}

This Assumption implies that problem \eqref{qp} has at least an optimal solution and $f^*$ is finite. As shown above, the CG method terminates at an optimal solution of \eqref{qp} in a finite number of iterations. We next show that the convergence of CG 
may depend on the eigenvalue distribution and a certain condition number of $A$. 
To this end,  let $\ell=\rank(A)$ and $\lambda_1\ge \lambda_2 \ge\cdots \ge\lambda_\ell>0$ be all nonzero eigenvalues of $A$. In addition, let $W\in \Re^{n\times \ell}$ and $V\in \Re^{n\times (n-\ell)}$
such that
$$A=[W \ V]\left[\ba{cc}
\hA  & 0 \\
0 & 0
\ea\right] \left[\ba{c}
W^T \\
V^T
\ea\right], \quad [W \ V]^T[W \ V]=I, \quad \hat{A}={\rm diag}(\lambda_1,\ldots,\lambda_\ell),$$
which is a standard eigenvalue decomposition of $A$. 
Clearly, $A=W\hat{A}W^T$ and $\hat{A}=W^TAW$. Moreover, it is not hard to observe that  
\beq \label{w-proj}
A=AWW^T, \quad\quad WW^T x = x \quad \forall x \in \range(A).
\eeq
Let $\{r^k\}$, $\{p^k\}$ and $\{x^k\}$ be generated by the above CG method. Define
\beq \label{wvec}
\hb = W^T b, \quad
\hr^k = W^T r^k, \quad \hp^k = W^T p^k, \quad \hx^k = W^T x^k.
\eeq

%We derive some useful properties for $\{\hr^k\}$, $\{\hx^k\}$ and $\{\hp^k\}$ as follows.

\begin{lemma} \label{prop-Wvec}
Let $\hA$, $\{\alpha_k\}$, $\{\beta_k\}$, $\{r^k\}$, $\{\hr^k\}$, $\{\hx^k\}$ and $\{\hp^k\}$
be defined above. Under Assumption \ref{assump-qp}, the following relations hold 
for any $k \ge 0$ such that $r^k \neq 0$:
\bi
\item[(i)] $\alpha_k = \frac{\|\hr^k\|^2}{(\hp^k)^T \hA \hp^k}$;
\item[(ii)] $\hx^{k+1} = \hx^k + \alpha_k \hp^k$;
\item[(iii)] $\hr^{k+1} = \hr^k + \alpha_k \hA \hp^k$;
\item[(iv)] $\beta_{k+1} = \frac{\|\hr^{k+1}\|^2}{\|\hr^k\|^2}$;
\item[(v)] $\hp^{k+1} = -\hr^{k+1} + \beta_{k+1} \hp^k$.
\ei
\end{lemma}

\begin{proof}
We observe from \eqref{ap-r} and Theorem \ref{cg-unbdd} (i) that under 
Assumption \ref{assump-qp},  $Ap^k \neq 0$ if and only if $r^k \neq 0$. It follows that if $r^k \neq 0$ for some $k\ge 0$, the CG method does not terminate at iteration $k$ and $(\alpha_k,x^{k+1},r^{k+1},\beta_{k+1},p^{k+1})$ must be generated. Thus 
$\hx^{k+1}$, $\hr^{k+1}$ and $\hp^{k+1}$ are well defined. 
It is known for the CG method that $r^k = Ax^k-b$, which along with 
Assumption \ref{assump-qp} implies that $r^k \in \range(A)$. It then follows from \eqref{w-proj} that $WW^T r^k = r^k$. 
%\[
%WW^T r^k = r^k, \quad WW^T p^k = p^k, \quad\quad \forall k \ge 0.
%\]
Using this relation, $A=W\hA W^T$ and \eqref{wvec}, we have
\beqa
&& \|\hr^k\|^2 = (W^Tr^k)^T (W^T r^k) = (r^k)^T (WW^T r^k) = (r^k)^T r^k = \|r^k\|^2, \label{hr} \\ [5pt]
&& (\hp^k)^T \hA \hp^k = (W^Tp^k)^T \hA  (W^T p^k) = (p^k)^T (W\hA W^T) p^k =  (p^k)^T A p^k, \nn
\eeqa
which together with \eqref{alphak} yields statement (i). Pre-multiplying
\eqref{xk} by $W^T$ and using \eqref{wvec}, one can see that statement (ii) holds. In addition, pre-multiplying \eqref{rk} by $W^T$ and using 
\eqref{w-proj}, \eqref{wvec} and $\hat{A}=W^TAW$, we have 
\[
\hr^{k+1} = W^T r^{k+1} = W^T r^k + \alpha_k W^T Ap^k = \hr^k + \alpha_k (W^T AW)(W^T p^k) = \hr^k + \alpha_k \hA \hp^k 
\]
and thus statement (iii) holds. Statement (iv) follows from \eqref{betak} and \eqref{hr}. Finally, statement (v) follows from
\eqref{wvec} and the pre-multiplication of \eqref{pk} by $W^T$.
\end{proof}

\gap

We are now ready to establish some further convergence properties for the CG method for solving
problem \eqref{qp}.

\begin{theorem} \label{CG-convergence}
Let $\{x^k\}$ be generated by the CG method applied to problem \eqref{qp} and 
$\ell=\rank(A)$. 
%Suppose that $\ell = \rank(A)$
%and $\lambda_1 \ge \lambda_2 \ge \cdots \ge \lambda_\ell >0$ are all nonzero eigenvalues of $A$.
 Under Assumption \ref{assump-qp}, the following properties hold:
\bi
%\item[(i)] The CG method terminates at an optimal solution of problem \eqref{qp} in at most $\ell$ iterations;
\item[(i)] If $A$ has $\hell$ distinct nonzero eigenvalues, the CG method terminates at an optimal
solution of problem \eqref{qp} in at most $\hell$ iterations;
\item[(ii)] $f(x^k) -f^* \le \left(\frac{\lambda_k-\lambda_\ell}{\lambda_k+\lambda_\ell}\right)^2 \left(f(x^0)-f^*\right)$ for all $1 \le k \le \ell$;
\item[(iii)] $f(x^k)-f^* \le 4 \left(\frac{\sqrt{\kappa(A)}-1}{\sqrt{\kappa(A)}+1}\right)^{2k}\left(f(x^0)-f^*\right)$ for all $k \ge 0$, where $\kappa(A)$ is defined in \eqref{kappa}. 
\ei
\end{theorem}

\begin{proof}
Let $W$, $\hA$, $\hb$, $\{\alpha_k\}$, $\{\hx^k\}$, $\{\hr^k\}$, $\{\beta_k\}$ and $\{\hp^k\}$
be defined above. 
%Recall that $W^TW=I$, $A=W\hat{A}W^T$ and $b\in $Range$(A)$. we have $A (I - WW^T)=0$ and $b^T (I - WW^T) =0$. These lead to
%\beqa
%A &=& A(WW^T + (I - WW^T)) = A WW^T , \label{Ax} \\ [5pt]
%b^T  &=& b^T (WW^T  + (I - WW^T)) = b^T WW^T. \label{bx}
%\eeqa
Recall that for the CG method, $Ap^k \neq 0$ if and only if $r^k\neq 0$. 
This along with \eqref{hr} implies that $Ap^k \neq 0$ if and only if $\hr^k\neq 0$. 
In view of \eqref{w-proj}, \eqref{wvec}, $\hA=W^TA W$ and $r^0=Ax^0-b$, one has
\[
\hr^0 = W^T r^0 = W^T (Ax^0-b) = W^T(A W W^T x^0 - b ) = (W^TA W) W^T x^0 - W^T b = \hA \hx^0 - \hb. 
\]
Using these and statements (i)-(v) of Lemma \ref{prop-Wvec}, we  observe that $\{\alpha_k\}$, $\{\hx^k\}$,
$\{\hr^k\}$, $\{\beta_k\}$ and $\{\hp^k\}$ can be viewed as the sequences generated by the CG method applied to
the problem
\beq \label{reduced-qp}
\hf^* = \min\limits_{\hx \in \Re^{\ell}} \hf(\hx) := \frac12 \hx^T \hA \hx - \hb^T \hx.
\eeq
Claim that the following relations hold for problems \eqref{qp} and \eqref{reduced-qp}.
\bi
\item[(a)] $f(x) = \hf(W^Tx)$ for all $x\in\Re^n$;
\item[(b)] $f(W\hx) = \hf(\hx)$ for all $\hx\in\Re^\ell$;
\item[(c)] $f^*=\hf^*$.
\ei
Indeed, since $b\in\range(A)$, it follows from \eqref{w-proj} that $WW^Tb=b$.  Using this, $A=W\hA W^T$ and \eqref{wvec}, we see that for any $x\in\Re^n$, 
\[
\ba{rcl}
(W^Tx)^T \hA (W^Tx)=x^T Ax, \quad \quad
\hb^T (W^Tx) = (b^TWW^T) x= b^T x,
\ea
\]
which along with the definitions of $f$ and $\hf$ imply that property (a) holds. In addition, by \eqref{wvec}, $\hA=W^TA W$ and the definitions of $f$
and $\hf$, one has for any $\hx\in\Re^\ell$,
\[
f(W\hx) = \frac12 \hx^T (W^T A W) \hx + (W^T b)^T \hx =  \frac12 \hx^T \hA \hx + \hb^T \hx = \hf(\hx)
\]
and hence property (b) holds. Combining properties (a) and (b), it is not 
hard to see that property (c) holds.

Notice that $\hA$ is symmetric positive definite. Thus $\{\hx^k\}$ is the sequence generated by the CG method for solving the strongly convex quadratic program \eqref{reduced-qp} with $\hx^0=W^Tx^0$. It follows that the
convergence results of the CG method for strongly convex QP, which are presented in \cite[Theorems 5.4 and 5.5 and equation (5.36)]{NoWr06}, hold for $\{\hx^k\}$. Let  $\hx^*$ be the unique optimal solution of \eqref{reduced-qp}. We are now ready to prove statements (i)-(iii) by using this observation and the above properties (a)-(c).

%(i) Notice that \cite[Theorem 5.3]{NoWr06} holds for $\{\hx^k\}$. One thus has $\hx^k = \hx^*$ for some
%$0 \le k \le \ell$, where $\hx^*$ is the unique optimal solution of \eqref{reduced-qp}. Hence $\hf(\hx^k)=\hf^*$. It then follows from $\hx^k=W^T x^k$ and properties (a) and (c) that
%\[
%f(x^k) = \hf(W^T x^k) = \hf(\hx^k)=\hf^* = f^*,
%\]
%which implies that $x^k$ is an optimal solution of \eqref{qp} and $r^\tk=Ax^\tk-b=0$. Hence, the CG 
%method terminates at an optimal solution of \eqref{qp} in at most $\ell$ iterations.

(i) Notice that \cite[Theorem 5.4]{NoWr06} holds for $\{\hx^k\}$. Then there exists some $0 \le \tk \le \hell$ such that $\hr^k=0$, which along with \eqref{hr} implies $r^k=0$. It follows that $Ap^k=0$ and $x^k$ is an optimal solution of \eqref{qp}. 
% Hence $\hf(\hx^k)=\hf^*$. It then follows from $\hx^k=W^T x^k$ and properties (a) and (c) that
%\[
%f(x^k) = \hf(W^T x^k) = \hf(\hx^k)=\hf^* = f^*,
%\]
%which implies that $x^k$ is an optimal solution of \eqref{qp}. 
%and $r^k=Ax^k-b=0$. The latter relation implies $Ap^k=0$. 
Hence, the CG method terminates at an optimal solution of \eqref{qp} in at most $\hell$ iterations.

(ii) %Since $A$ shares the same nonzero eigenvalues with $\hA$,  $\lambda_1 \ge \lambda_2 \ge \cdots \ge
%\lambda_\ell >0$ are also all nonzero eigenvalues of $\hA$.
Applying \cite[Theorem 5.5]{NoWr06} to problem \eqref{reduced-qp}, one has
\[
\|\hx^k -\hx^*\|^2_\hA \le \left(\frac{\lambda_k-\lambda_\ell}{\lambda_k+\lambda_\ell}\right)^2 \|\hx^0 -\hx^*\|^2_\hA, \quad \forall 1 \le k \le \ell,
\]
where $\|\hx\|_\hA=\sqrt{\hx^T \hA \hx}$ for any $\hx\in\Re^\ell$. Observe that 
$\hf(\hx) - \hf^* = \|\hx -\hx^*\|^2_\hA/2$ for all $\hx$. 
It thus follows that
\[
\hf(\hx^k) -\hf^* \le \left(\frac{\lambda_k-\lambda_\ell}{\lambda_k+\lambda_\ell}\right)^2 (\hf(\hx^0)-\hf^*), \quad
\forall \,\, 1 \le k \le \ell.
\]
 The conclusion of this statement follows from this relation, $\hx^k=W^T x^k$ and
the above properties (a) and (c).

(iii) Applying \cite[equation (5.36)]{NoWr06} to problem \eqref{reduced-qp}, we have
\[
\|\hx^k -\hx^*\|^2_\hA \le 4 \left(\frac{\sqrt{\kappa(\hA)}-1}{\sqrt{\kappa(\hA)}+1}\right)^{2k}\|\hx^0 -\hx^*\|^2_\hA, \quad \forall k \ge 0.
\]
Notice that $A$ shares the same nonzero eigenvalues with $\hA$. It follows from \eqref{kappa} that $\kappa(A)=\kappa(\hA)$. The conclusion of this statement then follows from these relations and a similar 
argument as for statement (ii).
\end{proof}

\section*{Acknowledgment}

The authors are grateful to Guoyin Li and Ting Kei Pong for bringing their attention to the reference \cite{Li95} that is greatly helpful in deriving the results of Section 
\ref{err-bdds}.

%%%%%%%%%%%%%%%%%%%%%%%%%%%%%%%%%%

\end{document}